\documentclass[a4paper,11pt]{article}

\usepackage[left=2.5cm,right=2.5cm,top=2cm,bottom=2cm]{geometry}

\usepackage{amsopn}
\usepackage[sort&compress,square,comma,numbers]{natbib}
\usepackage{amsfonts,amssymb,amsthm}
\usepackage[all]{xy} 
\usepackage{mathrsfs} 
\usepackage{url}
\usepackage{color}
\usepackage{enumerate}
\usepackage{blkarray}
\usepackage{subfig}
\usepackage[all]{xy}
\usepackage{graphicx,xcolor}
\usepackage{hyperref}
\usepackage[version=3]{mhchem}
\usepackage{tikz}
\usetikzlibrary{decorations}
\usetikzlibrary{shapes}
\usepackage{relsize}

\newcommand{\mmN}{\mathcal{N}}
\newcommand{\mmE}{\mathcal{E}}

\newcommand{\mmG}{\mathcal{G}}
\renewcommand{\k}{\kappa}
\newcommand{\B}{\mathcal{B}}

\newcommand{\R}{{\mathbb R}}
\newcommand{\Rnn}{\mathbb{R}_{\geq 0}}

\newcommand{\Rnng}{R_{\geq 0}}

\newcommand{\st}{\,\mid \,}
\DeclareMathOperator{\im}{im} 
\DeclareMathOperator{\supp}{supp}

\newtheorem{lemma}{Lemma}
\newtheorem{proposition}{Proposition}
\newtheorem{theorem}{Theorem}
\newtheorem{corollary}{Corollary}

\theoremstyle{definition}
\newtheorem{remark}{Remark}

\newtheorem{definition}{Definition}
\newtheorem{example}{Example}

\begin{document}

\title{Graphical criteria for positive solutions to linear systems}

\author{
Meritxell S\'aez$^1$, Elisenda Feliu$^1$, and Carsten Wiuf$^1$
}

\footnotetext[1]{Department of Mathematical Sciences, University of Copenhagen, Universitetsparken 5, 2100 Copenhagen, Denmark}
\footnotetext[2]{This work was funded by the Danish Research Council and the Lundbeck Foundation, Denmark} 

 \date{\today}

\maketitle

\begin{abstract}
We study linear systems of equations with coefficients in a generic partially ordered ring $R$ and a unique solution, and seek conditions  for the solution to be nonnegative,  that is, every component of the solution is a quotient of two nonnegative elements in $R$. The requirement of a nonnegative solution arises typically  in  applications, such as in biology and ecology, where quantities of interest are concentrations and abundances. We provide novel conditions on a labeled multidigraph  associated with the linear system that guarantee   the solution to be nonnegative. Furthermore, we study a generalization of the first class of linear systems,  where  the coefficient matrix has a specific block form and  provide analogous conditions for nonnegativity of the solution, similarly based  on a labeled multidigraph. The latter scenario arises naturally in   chemical reaction network theory, when studying full or partial parameterizations of the positive part of the steady state variety of a  polynomial dynamical system in the concentrations of the molecular species.

\end{abstract}

\section{Introduction}
A classical problem in applied mathematics is to determine the solutions to a linear system of equations. In applications, it is often the case  that only positive or nonnegative real solutions to the system are meaningful, and  criteria to assert positivity and nonnegativity of the solutions  have thus been developed \cite{Dines26,Kay85,FarRin2000,Roman2005}.  Nonnegativity is  required, for example, in the case of equilibria concentrations of molecular species in biochemistry \cite{feinbergnotes,Fel_elim,gunawardena-notes}, species abundances at steady state in ecology \cite{may}, stationary distributions of Markov chains in probability theory \cite{norris}, and  in Birch's theorem for maximum likelihood estimation in statistics \cite{Pachter-Sturmfels3}.  Also in economics and game theory are equilibria often required to be nonnegative \cite{Judd12,Gand09}.
Many of these situations arise from considering dynamical systems where the state variables are restricted to the positive or nonnegative orthant.

We consider a linear system of equations $Ax+b=0$,  with $\det A\not=0$ and where the coefficients of $A,b$ are in a generic partially ordered ring $R$. By adding an extra row to $A$ such that the column sums are zero, we  might associate in a natural way  a Laplacian $L$ with the linear system and the corresponding (so-called) labeled canonical  multidigraph. The components of the solution to the linear system are rational functions on the entries of $A$ and $b$. Their numerator and denominator can, by means of the Matrix-Tree Theorem \cite{Tutte-matrixtree}, be expressed as polynomials on the labels of the rooted spanning trees of the canonical multidigraph, and in fact, of any labeled multidigraph with Laplacian $L$. If the multidigraph is what we call a \emph{P-graph} (Definition~\ref{def:Pgraph}), then we show that the solution to $Ax+b=0$ is nonnegative. These conditions  are readily fulfilled if the off-diagonal elements of $L$ (not only $A$) are nonnegative.

In the applications motivating this work, the P-graph condition is hardly met. This happens for example in the study of biochemical reaction networks, where the system of interest arises from computing the equilibrium points of a dynamical system constrained to certain invariant linear varieties. However, in many instances the solution is nevertheless nonnegative. In the second part of the paper we explore this other scenario. We consider systems  of a specific block form, 
compatible with the application setting, and derive conditions on another (related) multidigraph that ensure nonnegativity of the solution. In this situation, the solution is expressed as a rational function in the labels of rooted spanning forests of the multidigraph. The second scenario is an extension of the generic case given in the first part of the paper.

Typically in applications,  the entries of $A,b$ depend on  parameters and inputs 
 that cannot be fixed beforehand, but must be treated as ``unknown'' or symbolic variables. Our approach  accommodates this since solutions are given as rational functions in these entries. Alternatively, one might view the parameters as functions, and apply the results for the ring of real valued functions. In this way, we can study the nonnegativity of solutions without fixing parameters and inputs.

One natural application of our results, which motivated this work, is within biochemical reaction network theory and concerns the parameterization of the positive part of the algebraic variety of steady states. The concentrations of the molecular species in the reactions evolve according to a non-linear ODE system of equations and the steady states are found by equating the equations of the ODE system to zero \cite{Fel_elim}. Our approach is tailored to obtain a full or partial parameterization of the set of positive steady states, possibly constrained  to linear invariant varieties, in terms of the parameters of the system and some variables. These parameterizations can be obtained for a large class of non-linear reaction networks. The results given here generalise and complement earlier strategies in finding these parameterizations by means of linear elimination \cite{gunawardena-linear,TG-rational,fwptm,Fel_elim}. We give an example towards the end of the paper.

The paper is organized as follows. 
In Section \ref{sec:prelim} we introduce  notation, background material and present the solution to a linear system in terms of the spanning trees of an associated multidigraph. In Section \ref{sec:possolution} we give conditions in terms of the associated multidigraph to decide the nonnegativity of a solution. In Section \ref{sec:secondsystem} we develop theory for the second scenario and provide examples. Finally, in Section \ref{sec:proofs}, we prove the two main results of Section \ref{sec:secondsystem}.

\section{Preliminaries} \label{sec:prelim}

Let $\R_{\geq 0}$ and $\R_{>0}$ be  the sets of nonnegative and positive real numbers, respectively.

For a ring $R$, let $R^n$ be the $n$-tuples of elements in $R$ and  $R^{n\times m}$ be the $n$ times $m$ matrices with entries in $R$.  If $R$ is a partially ordered  ring, then the notions of positive, negative, nonpositive and nonnegative elements are well defined  \cite{NS_notes}. These elements form the sets $R_{>0}, R_{<0}, R_{\leq 0}, R_{\geq 0}$, respectively. Some examples of  $R$ are $\mathbb{Q}$, $\R$ and the real functions defined on a domain $\Omega$ ordered by pointwise comparison.
 
The support of an  element $x\in R^n$ is defined as $\supp(x)=\{i\mid x_i\neq 0\}.$

The cardinality of a finite set $F$ is denoted by $|F|$, the  power set by $\mathcal{P}(F)$ and the disjoint union of two sets $E,F$  by  $E\sqcup F$. For finite pairwise disjoint sets  $F_1,\dots,F_k$, we define the following set of  unordered $k$-tuples of $F_1\cup \dots \cup F_k$:
\[
F_1\odot \cdots \odot F_k=\underset{i\in\{1,\dots,k\}}{\mathlarger{\mathlarger\odot}}\!\!F_i = \Big\{\{w_1,\dots,w_k\}\st w_j\in F_j\text{ for }j=1,\dots,k\Big\}.
\]
The $k$-tuples of $F_1\odot \cdots \odot F_k$ have $k$ elements.

\subsection{Multidigraphs and the Matrix-Tree Theorem}
In this subsection we introduce  notation and concepts related to algebraic graph theory.
A \emph{multidigraph} $\mmG$ is a pair of finite sets $(\mmN, \mathcal{E})$, called the set of nodes and the set of edges, respectively, 
together  with two functions, $s\colon\mathcal{E} \rightarrow \mmN$ and $t\colon\mathcal{E} \rightarrow \mmN$, called the source and the target function, respectively.
The function $s$ (resp. $t$) assigns to each edge the source (resp. target) node of the edge. An edge $e$ is a self-edge if $t(e)=s(e)$, and two edges $e_1,e_2$ are parallel edges if $t(e_1)=t(e_2)$ and $s(e_1)=s(e_2)$. 

A \textbf{cycle} is a closed \emph{directed} path with no repeated nodes. 
A \textbf{tree} $\tau$ is a directed subgraph of $\mmG$ such that the underlying undirected graph is connected and acyclic. A tree $\tau$ is \emph{rooted} at the node $N$, if $N$ is the only node without outgoing edges. In that case, there is a unique directed path from every node in $\tau$ to $N$. A \textbf{forest} $\zeta$ is a directed subgraph of $\mmG$ whose connected components are trees. A tree (resp.~forest) is called a \emph{spanning tree} (resp.~\emph{spanning forest}) if its node set is $\mmN$. 
For a spanning tree $\tau$ (resp.~a spanning forest $\zeta$) we use $\tau$ (resp.~$\zeta$) to refer to the edge set of the graph and to the graph itself indistinctly, as the node set in this case is $\mmN$.  The number of edges of a spanning forest $\zeta$ is $|\mmN|$ minus the number of connected components of $\zeta$. 

If $\pi\colon\mmE \rightarrow R$ is a labeling of $\mmG$ with values in a ring $R$, then any sub-multidigraph $\mmG'$ of $\mmG$ inherits a labeling from $\mmG$.  A labeling  is extended to  $\mathcal{P}(\mmE)$ by
\[
\pi\colon\mathcal{P}(\mmE)\to R,\quad \pi(\mmE')=\prod_{e\in\mmE'}\pi(e) \quad \text{for}\quad\mmE'\subseteq \mmE.
\]

In the following we assume that $\mathcal{G}=(\mmN, \mmE)$ is a (labeled) multidigraph with no self-loops and node set $\mmN=\{1,\dots,m+1\}$ for some $m\ge 0$.
For two sets $F,B\subseteq \mmN$ with $|F|=|B|$, let $\boldsymbol{\Theta_\mmG(F,B)}$ be the set of spanning forests of $\mmG$ such that:
\begin{enumerate}[(i)]
\item each forest has $|B|$ connected components (trees),
\item each tree contains a node  in $F$ and is rooted at a node in $B$.
\end{enumerate}
Each forest $\zeta\in\Theta_\mmG(F,B)$ induces a bijection $g_\zeta\colon F\rightarrow B$ with $g_\zeta(i)=j$ if $j$ is the root of the tree containing $i$. A pair $(i_1,i_2)\in F\times F$ is an \emph{inversion} in $g_\zeta$ if $i_1<i_2$ and $g_\zeta(i_1)>g_\zeta(i_2)$. We denote by $I(g_\zeta)$ the number of inversions in $g_\zeta$ and define
\begin{equation}\label{ypsilon}
\widetilde{\Upsilon}_\mmG(F,B)=\sum_{\zeta\in \Theta_\mmG(F,B)}(-1)^{I(g_\zeta)}\pi(\zeta)\quad \text{    and }\quad\Upsilon_\mmG(F,B)=\sum_{\zeta\in \Theta_\mmG(F,B)}\pi(\zeta),
\end{equation}
where the empty sum is defined as zero. 

Let $\mmE_{ji}=\{e\in\mmE\st s(e)=j,\ t(e)=i\}$ be the set of parallel edges with source $j$ and target $i$. 
The \textbf{Laplacian} of $\mathcal{G}$ is the $(m+1)\times (m+1)$ matrix $L=(L_{ij})$ with
\[
L_{ij}= \sum\limits_{e\in \mmE_{ji}}\pi(e) \quad \text{for }i\neq j, \quad\text{ and }\quad
L_{ii}=-\sum\limits_{k\neq i}L_{ki}.
\]

The column sums of the Laplacian of $\mathcal{G}$ are zero by construction. Any square matrix $L\in R^{(m+1)\times (m+1)}$ with zero column sums can be realized as the Laplacian  of a labeled multidigraph. If this is the case we say that $L$ is a Laplacian. The \textbf{canonical multidigraph}  with Laplacian $L$ is defined as the labeled multidigraph  with  node set $\mmN=\{1,\dots,m+1\}$ and one edge $j\to i$ with label $L_{ij}$ for each nonzero entry $L_{ij}\neq 0$,  for $i\not=j$. This multidigraph has neither parallel edges nor self-loops, thus it is a digraph.  All other labeled  multidigraphs with the same Laplacian can be obtained from the canonical multidigraph by adding self-edges and splitting edges into parallel edges while preserving the  label sums. 

\begin{theorem}{({\bf All Minors Matrix-Tree Theorem} \cite[Th 3.1]{Moon94})}\label{thm:matrixtree}
Let $L$ be the Laplacian  of a   labeled   multidigraph $\mathcal{G}$ with $m+1$ nodes and let $B,F\subseteq \mmN$ be such that $|B|=|F|$. Let $L_{(F,B)}$ be the minor obtained from $L$ by removing the rows with index in $F$ and the columns with index in $B$.  Then
\[
L_{(F,B)}=(-1)^{\epsilon}\ \widetilde{\Upsilon}_\mmG(F,B),\quad \text{where}\quad\epsilon= m+1-|F|+\sum_{i\in F}i+\sum_{j\in B}j.
\]
\end{theorem}

The All Minors Matrix-Tree Theorem is usually stated for digraphs. Using Lemma 1 in \cite{Saez:reduction} it holds also for multidigraphs.  When $|B|=|F|=1$, Theorem \ref{thm:matrixtree} is the usual Matrix-Tree Theorem  extended to multidigraphs \cite{Tutte-matrixtree}. 

Define 
\[
\Theta_\mmG(B)=\bigcup\limits_{\begin{subarray}{c}F\subseteq \{1,\dots,m+1\}\\ |B|=|F| \end{subarray}} \Theta_\mmG(F,B) \quad \text{ and }\quad \Upsilon_\mmG(B)=\sum\limits_{\zeta\in\Theta_\mmG(B)} \pi(\zeta).
\]
The set $\Theta_\mmG(B)$ consists of spanning forests of $\mmG$ with $|B|$ connected components and each component is a tree rooted at a node   in $B$.

If $i,j\in\mmN$,   then $\Theta_\mmG(\{i\},\{j\})$ is the set of spanning trees rooted at $j$, hence it is independent of $i$.  
We denote the set by $\Theta_\mmG(j)$ and  let $\Upsilon_\mmG(j)=\sum_{\tau\in\Theta_\mmG(j)}\pi(\tau).$  
For  $\tau\in\Theta_\mmG(j)$, we have $I(g_\tau)=0$. If $\mmG$ is not connected, then $\Theta_\mmG(j)=\emptyset$ and $\Upsilon_\mmG(j)=0$.

\subsection{Linear systems}

Let $R$ be a ring. Consider a linear system $Ax+b=0$, where $A=(a_{ij})\in R^{m\times m}$ is a nonsingular  matrix, $b\in R^m$ is a vector of independent terms, and $x$ is an $m$-dimensional vector of unknowns.  Define the $(m+1)\times (m+1)$ matrix $L$ by
\begin{equation}\label{eq:Laplacian}
L=\left(\begin{array}{c|c}
A & b\\ \hline
\cdots & \cdot
\end{array} \right), \quad\text{where}\quad  L_{ij}=\begin{cases} 
\quad a_{ij} & \text{for }i,j\leq m\\
\quad b_i & \text{for }i\leq m \text{ and }j=m+1\\
-\sum\limits_{k=1}^m a_{kj} & \text{for } i=m+1\text{ and }j\leq m   \\
-\sum\limits_{k=1}^m b_{k} & \text{for } i=j=m+1,
\end{cases}
\end{equation}
such that $L$ has zero column sums. Therefore, $L$ is the Laplacian of a  labeled multidigraph $\mathcal{G}$ with $m+1$ nodes.

The solution to the linear system $Ax+b=0$ can be expressed in terms of the labels of the spanning trees of any  labeled multidigraph with Laplacian $L$ as follows. Let $A_{i\rightarrow b}$ be the matrix obtained by replacing the $i$-th column of $A$ with the vector $b$, and let $A_{i}| b$ be the matrix obtained by removing the $i$-th column of $A$ and taking $b$ as the $m$-th column. 

\begin{proposition}\label{prop:solCMT}
Let $A\in R^{m\times m}$, $b\in R^m$ and $x=(x_1,\dots,x_m)$. Let $L$ be as in \eqref{eq:Laplacian} and $\mmG$ a  labeled multidigraph with Laplacian $L$. 
Then, 
$\det(A)=(-1)^{m}\Upsilon_\mmG(m+1)$. Further, if  $\det(A)\neq 0$, then the solution to the linear system $Ax+b=0$ is 
\[
x_i=\dfrac{\Upsilon_\mmG(i)}{\Upsilon_\mmG(m+1)}\,\,\in\,\,\widehat{R},\,\qquad i=1,\dots,m,
\]
where $\widehat{R}\supseteq R$ is an extension ring of $R$ for which the above quotient is defined.
\end{proposition}

\begin{proof}
By Theorem \ref{thm:matrixtree} we have 
\begin{align*}
\det(A)&=L_{(\{m+1\},\{m+1\})}=(-1)^{m+1-1+m+1+m+1}\Upsilon_\mmG(m+1)=(-1)^m\Upsilon_\mmG(m+1),  \\
\det(A_{i}| b)&=L_{(\{m+1\},\{i\})}=(-1)^{m+1-1+m+1+i}\Upsilon_\mmG(i)=(-1)^{i+1}\Upsilon_\mmG(i),
\end{align*}
 as the number of inversions is  $0$ in these two cases.
By Cramer's rule we have 
\[
x_i=\dfrac{\det(A_{i\rightarrow -b})}{\det(A)}=
\dfrac{-\det(A_{i\rightarrow b})}{\det(A)}=
\dfrac{(-1)^{m-i+1}\det(A_{i}| b)}{\det(A)}=\dfrac{(-1)^m\Upsilon_\mmG(i)}{(-1)^m\Upsilon_\mmG(m+1)}.
\]
\end{proof}

\refstepcounter{theorem}\label{ref:split2example}
\newcounter{mycounterSplit}
\renewcommand{\themycounterSplit}{\getrefnumber{ref:split2example}\,(part\,\Alph{mycounterSplit})}
\newtheorem{myexampleSplit}[mycounterSplit]{Example}

\begin{myexampleSplit}\label{ex:split2}
Let $z_1,\dots,z_5\in R\setminus \{0\}$.  Consider the linear system  with $m=3$,
\[
\left(\begin{array}{ccc}
-z_2 & \phantom{-}0 & \phantom{-}z_4 \\ 
-z_1 & -z_3 & \phantom{-}0  \\ 
-z_2 & \phantom{-}z_3 & -z_4 \\
\end{array} \right)\hspace{-0.1cm}
\left(\begin{matrix}x_1\\ x_2\\ x_3 \end{matrix} \right) +
\left(\begin{matrix}0\\ z_5\\ 0 \end{matrix} \right)=
\left(\begin{matrix}0\\ 0\\ 0 \end{matrix} \right).
\]
The matrix $L$ and the corresponding canonical multidigraph are

\begin{minipage}{0.52\linewidth}
\[
L=\left(\begin{array}{ccc|c}
-z_2 & 0 & \phantom{-}z_4 &  0\\ 
-z_1 & -z_3 & 0 &z_{5} \\ 
-z_2 & \phantom{-}z_3 & -z_4 &0\\
\hline
z_1+2z_2 & 0 & 0 & -z_{5}
\end{array} \right)
\]
\end{minipage}
\begin{minipage}{0.4\linewidth}
\begin{center}
\begin{tikzpicture}[inner sep=1.2pt]
\node (v31) [shape=circle] at (-2,0) {$3$};
\node (v11) [shape=circle] at (0,0) {$1$};
\node (v21) [shape=circle] at (2,0) {$2$.};
\node (*1) [shape=circle] at (0,-1.5) {$4$};

\draw[->] (v11) to[out=0,in=180] node[above,sloped]{\footnotesize $-z_1$}(v21);
\draw[->] (v11) to[out=170,in=10] node[above,sloped]{\footnotesize $-z_2$}(v31);
\draw[->] (v11) to[out=270,in=90] node[left]{\footnotesize $z_1+2z_2$}(*1);
\draw[->] (v21) to[out=155,in=25] node[above,sloped]{\footnotesize $z_3$}(v31);
\draw[->] (*1) to[out=20,in=240] node[right]{\footnotesize $\ z_{5}$}(v21);
\draw[->] (v31) to[out=-10,in=190] node[below,sloped]{\footnotesize $z_4$}(v11);
\end{tikzpicture}
\end{center}
\end{minipage}

\medskip
\noindent Therefore, $\Upsilon_\mmG(2)=\left(z_1+ 2\,z_{{2}}\right) z_4 z_5-z_1z_4z_5=2z_2z_4z_5$ and $\Upsilon_\mmG(4)=-\det(A)=\left( z_{{1}}+2\,z_{{2}} \right)z_3z_4$. The terms $\Upsilon_\mmG(1)$ and $\Upsilon_\mmG(3)$ are similarly found to obtain the solution
\[
x_1={\frac {z_{{5}}}{z_{{1}}+2\,z_{{2}}}},\qquad 
x_2={\frac { 2\,z_{{2}}  z_5}{\left( z_{{1}}+2\,z_{{2}} \right)z_3 }},\qquad 
x_3={\frac { z_{{2}}z_5 }{ \left( z_{{1}}+2\,z_{{2}}\right) z_{{4}}}}.
\]  
\end{myexampleSplit}

\section{Positive solution to a linear system}\label{sec:possolution}

Let $R$ be a partially ordered ring. We are interested in conditions that ensure the  solution to the linear system $Ax+b=0$ in $\widehat{R}^m$  is nonnegative (cf. Proposition~\ref{prop:solCMT}).
If the off-diagonal entries of $L$ are  in $\Rnng$, then this is always the case. Indeed, the edge labels  
of the canonical multidigraph $\mathcal{G}$ with Laplacian  $L$ are in $\Rnng$. Hence by Proposition \ref{prop:solCMT}  and the definition of $\Upsilon_\mmG(j)$, the solution is in $\widehat{R}^m_{\geq 0}$. 

\smallskip
This condition is  however  not necessary.  Consider  Example \ref{ex:split2} with $R=\R$ and $z_i\geq 0$. Not all off-diagonal entries of $L$ are in $\Rnn$, but the  solution is nonetheless in $\Rnn^3$ if $\det(A)\neq 0$. 

In the following we consider labeled multidigraphs with no zero labels ($\pi(e)\not=0$ for all $e\in\mmE$) and such that  the label of each edge is either positive or negative.  In this case,  the labels of the spanning trees are also either nonnegative or nonpositive. 
Assuming $\mmG$ is a P-graph (Definition \ref{def:Pgraph} below), we will show that any nonpositive term in $\Upsilon_\mmG(i)$ corresponding to a spanning tree with nonpositive label cancels with a sum of labels of spanning trees with nonnegative labels.
Definition \ref{def:Pgraph}  below guarantees  that any nonpositive label of a spanning tree  rooted at $i$ cancels out in  $\Upsilon_\mmG(i)$ with a sum of labels of spanning trees with nonnegative labels. Further,  $\Upsilon_\mmG(i)\in\Rnng$ and the solution to the linear system given in Proposition \ref{prop:solCMT} is  in  $\widehat{R}^m_{\geq 0}$.    This is what happens in Example \ref{ex:split2}.

For a multidigraph $\mmG=(\mmN,\mmE)$   with labeling $\pi\colon \mmE\rightarrow R$, we let
 \[
\mmE^-=\big\{e\in\mmE\st \pi(e)\in R_{>0}\big\}\quad\text{and}\quad \mmE^+=\big\{e\in\mmE\st \pi(e)\in R_{<0}\big\}
\] 
denote the set of edges with positive and negative labels, respectively.

\begin{definition}\label{def:Pgraph}
Let $\mmG=(\mmN,\mmE)$ be a multidigraph with labeling $\pi\colon \mmE\rightarrow R$. Let $\mu\colon \mmE^-\rightarrow \mathcal{P}(\mathcal{E}^+)$ be a map. The pair $(\mmG,\mu)$ is an  \textbf{edge partition} if  
\begin{enumerate}[(i)]
\item \label{cond1Pgraph} $\mmE=\mmE^+\sqcup \mmE^-$.
\item\label{cond2Pgraph} All cycles in $\mmG$ contain at most one edge in $\mmE^-$.
\item The map $\mu$ is such that for every $e\in\mmE^-$
\begin{enumerate}[(a)]
\item\label{cond3aPgraph} if $e'\in\mu(e)$, then $s(e)=s(e')$,
\item\label{cond3bPgraph}  if $e'\in\mu(e)$, then every cycle containing $e'$ contains $t(e)$,
\item\label{cond3cPgraph} if $e\neq e'$, then $\mu(e)\cap\mu(e')=\emptyset$. 
\end{enumerate}
\end{enumerate} 

\medskip
We say $\mmG$  a \textbf{P-graph} if there is an edge partition $(\mmG,\mu)$ such that  for every $e\in\mmE^-$,
\begin{enumerate}[(iv)]
\item \label{cond4Pgraph} $\pi(e)+\sum\limits_{e'\in\mu(e)}\pi(e')\in\Rnng$. 
\end{enumerate} 
In this case we say that the map $\mu$ is \emph{associated} with the P-graph $\mmG$.
\end{definition}

Defintion \ref{def:Pgraph}(\ref{cond3bPgraph}) is equivalent to the condition that any path from $t(e')$ to $s(e')=s(e)$ contains $t(e)$. If the labels of a labeled multidigraph $\mmG$ are all positive, then $\mmG$ is trivially a P-graph.  Note that there does not necessarily  exist a P-graph with a given Laplacian.

\begin{myexampleSplit}\label{ex:split3}
Consider the multidigraph $\mmG$ with Laplacian $L$ in Example \ref{ex:split2}. Since there is only one edge in $\mmE^+$ with source $1$ but two edges in $\mmE^-$ with the same source, condition (\hyperref[cond4Pgraph]{iv}) is not satisfied for any choice of $\mu$. Hence  $\mmG$ is not a P-graph.  The following multidigraph is a P-graph with Laplacian $L$. An associated map $\mu$ is given.

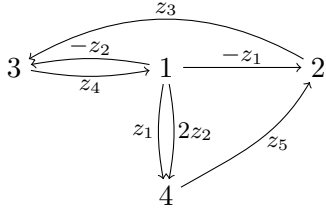
\begin{minipage}{0.45\linewidth}
\begin{center}
\begin{tikzpicture}[inner sep=1.2pt]
\clip(-2.5,-2) rectangle (2.5,1);
\node (v31) [shape=circle] at (-2,0) {$3$};
\node (v11) [shape=circle] at (0,0) {$1$};
\node (v21) [shape=circle] at (2,0) {$2$};
\node (*1) [shape=circle] at (0,-1.7) {$4$};

\draw[->] (v11) to[out=0,in=180] node[above,sloped]{\footnotesize $-z_1$}(v21);
\draw[->] (v11) to[out=170,in=10] node[above,sloped]{\footnotesize $-z_2$}(v31);
\draw[->] (v11) to[out=260,in=100] node[left]{\footnotesize $z_1$}(*1);
\draw[->] (v11) to[out=280,in=80] node[right]{\footnotesize $2z_2\ $}(*1);
\draw[->] (v21) to[out=150,in=30] node[above,sloped]{\footnotesize $z_3$}(v31);
\draw[->] (*1) to[out=30,in=240] node[right]{\footnotesize $\ z_{5}$}(v21);
\draw[->] (v31) to[out=-10,in=190] node[below,sloped]{\footnotesize $z_4$}(v11);
\end{tikzpicture}
\end{center}
\end{minipage}
\begin{minipage}{0.45\linewidth}
\vspace{-30pt}
\[
\mu(1\ce{->[-z_1]}2)=\{1\ce{->[z_1]} 4 \},
\]
\[
 \mu(1\ce{->[-z_2]}3)=\{1\ce{->[2z_2]} 4 \}.
\]
\end{minipage}

\noindent
The entry $z_1+2z_2$ in $L$ is made into two  edges in the multidigraph. This is necessary in order to satisfy Definition \ref{def:Pgraph}\eqref{cond3cPgraph} and (\hyperref[cond4Pgraph]{iv})  simultaneously.
\end{myexampleSplit}

For some rings $R$, such as  the ring of real functions over a real domain, it is possible to write any element as the sum of a positive and a negative element. In this case, given
 a labeled multidigraph with Laplacian $L$, a new labeled multidigraph with the same Laplacian can be made by splitting each edge (if necessary) into two parallel edges, one with  positive label and one with negative label.  Hence Definition \ref{def:Pgraph}\eqref{cond1Pgraph} can always be fulfilled. The key part of Definition \ref{def:Pgraph} is the simultaneous fulfilment of  condition (\ref{cond3cPgraph}) and (\hyperref[cond4Pgraph]{iv}).

\begin{remark}\label{rmk:condP}
It follows from Definition \ref{def:Pgraph}(\hyperref[cond4Pgraph]{iv}), that a necessary condition for a multidigraph to be a P-graph is that the label sums over the edges with the same source are nonnegative, that is, the diagonal entries of the Laplacian are nonpositive. 
\end{remark}

\smallskip
Let $(\mmG,\mu)$ be an edge partition and define the set $\im\mu$ by
\[
\im\mu=\bigcup_{e\in\mmE^-}\mu(e)\subseteq \mmE^+.
\]
By Definition~\ref{def:Pgraph}(\ref{cond3cPgraph}), each edge $e\in\im\mu$ belongs to exactly one set $\mu(e')$.  We might thus define the ``inverse'' of $\mu$ by
$$
\mu^{*}\colon\im\mu\rightarrow \mmE^-, \quad  \mu^{*}(e')=e\ \text{ if }\ e'\in\mu(e).
$$

We will show that the spanning forests in $\Theta_\mmG(B)$, $B\subseteq \mmN$,  can  be obtained from a smaller set of spanning forests $\Lambda_\mmG(B)\subseteq \Theta_\mmG(B)$ by replacing edges in $\mmE^-$ with edges in $\im\mu\subseteq \mmE^+$.   The spanning forests of $\Lambda_\mmG(B)$ are characterized by having as many edges as possible in $\mmE^-$, in a sense that will be made precise in Lemma \ref{lemma:maxEzeta}. This further allows us to characterize the labels of the spanning forests in $\Theta_\mmG(B)$ in terms of the associated map $\mu$, see Lemma \ref{lemma:decompPrePsi} and Theorem \ref{thm:posol} below.  

In the next two lemmas we will make use of the following fact. Given a spanning forest $\zeta\in \Theta_\mmG(B)$, let $\zeta'$ be a submultidigraph obtained by replacing some edges of $\zeta$ by other edges with the same source. We claim that  if $\zeta'$ does not contain any cycle, then  it belongs to $\Theta_\mmG(B)$.  Indeed, if this is so, then $\zeta'$ has $|B|$ connected components since it has $m-|B|$ edges and is a spanning forest. Further, by construction, there is not an edge with  source in $B$. Hence each connected component is a tree rooted at a node   in $B$.

 Recall that  we identify a spanning tree $\zeta$  with its set of edges, and thus $\zeta\cap \im\mu$ denotes the set of edges of $\zeta$ in $\im \mu$. 

\begin{lemma}\label{lemma:maxEzeta}
Let $(\mmG,\mu)$ be an edge partition, 
 $B\subseteq \mmN$, $\zeta\in \Theta_\mmG(B)$, and  
\[
\mmE_\zeta=\Big\{E\subseteq \zeta\cap \im\mu \mid (\zeta\setminus E)\cup \mu^{*}(E)\in \Theta_\mmG(B)\Big\}\, \subseteq \mathcal{P}(\mmE^+).
\]
The set $\mmE_\zeta$ is closed under union, that is, if $E_1,E_2\in \mmE_\zeta$, then $E_1\cup E_2\in \mmE_\zeta$. 
\end{lemma}

\begin{proof} 
Consider distinct $E_1,E_2\in \mmE_\zeta$ and let $$\zeta_i=(\zeta\setminus E_i)\cup \mu^{*}(E_i)\in\Theta_\mmG(B)$$ for $i=1,2$. We claim that $E_3=E_1\cup E_2\in\mmE_\zeta$, that is, $\zeta_3=(\zeta\setminus E_3)\cup \mu^{*}(E_3)\in\Theta_\mmG(B)$. Since the image of $\mu^{*}$ belongs to $\mmE^-$, edges with positive label in $\zeta_3$ are also edges in $\zeta_1$ and $\zeta_2$ by construction.  Using that a cycle in $\mmG$ contains at most one edge in $\mmE^-$  by Definition  \ref{def:Pgraph}\eqref{cond2Pgraph}, we conclude that any cycle  in $\zeta_3$ is also  a cycle of $\zeta_1$ or $\zeta_2$. However, they do not contain  cycles as they are forests. By the argument above, this implies that $\zeta_3$ is a spanning forest in $\Theta_\mmG(B)$.
\end{proof}

Since $\zeta\cap \im\mu$ is finite, it follows from the lemma that $\mmE_\zeta$ has a unique maximum with respect to inclusion. That is, there is a set $E_\zeta\in\mmE_\zeta$ such that for all $E\in\mmE_\zeta$ it holds $E\subseteq  E_\zeta$. 

 Let $(\mmG,\mu)$ be an edge partition. For $B\subseteq \mmN$,  define the set
\begin{equation*}
\Lambda_\mmG(B)=\big\lbrace \zeta \in \Theta_\mmG(B)\st \mmE_\zeta=\emptyset\big\rbrace,
\end{equation*}
which consists of the spanning forests  that are maximal with respect to the edge replacement operation defined in Lemma \ref{lemma:maxEzeta} and thus have  the maximal number of negative labels. Note that we suppress the dependence of $\mu$ in $\Lambda_\mmG(B)$.
We define a  surjective map  by
$$\psi\colon \Theta_\mmG(B)\rightarrow\Lambda_\mmG(B),\quad  \psi(\zeta)=\big(\zeta\setminus E_\zeta\big)\cup \mu^{*}\big(E_\zeta\big).$$
This gives a partition of $\Theta_\mmG(B)$,
\begin{equation}\label{eq:decompTheta}
\Theta_\mmG(B)=\bigsqcup\limits_{\zeta\in\Lambda_\mmG(B)}\psi^{-1}(\zeta).
\end{equation}

\begin{lemma}\label{lemma:decompPrePsiEdges}
Let $(\mmG,\mu)$ be an edge partition,  $B\subseteq \mmN$, $\zeta\in\Theta_\mmG(B)$, $e\in\zeta\cap\mmE^-$ and $e'\in\mu(e)$. 
Then $$\zeta'=(\zeta\setminus \{e\})\cup \{e'\}\in\Theta_\mmG(B).$$
\end{lemma}
\begin{proof}
By assumption $\zeta$ does not contain cycles. If $\zeta'$ contains a cycle, then it contains $e'$ and $t(e)$ by Definition \ref{def:Pgraph}\eqref{cond3bPgraph}. It implies that there is a path from $t(e)$ to $s(e')=s(e)$ in $\zeta$ as well. 
Since $\zeta$ contains $e$, there is also a cycle in $\zeta$, and we have reached a contradiction. Hence, $\zeta'\in\Theta_\mmG(B)$.
\end{proof}

\begin{lemma}\label{lemma:decompPrePsi}
With the notation introduced above, it holds that for $\zeta\in\Lambda_\mmG(B)$,
\[
\psi^{-1}(\zeta)=\left\{ (\zeta\cap \mmE^+) \cup E \st E\in  \underset{e\in\zeta\cap \mmE^-}{\mathlarger{\mathlarger\odot}} \big(\{e\}\cup \mu(e)\big)   \right\}.
\]
\end{lemma}

\begin{proof} 
If $\psi(\zeta')= \zeta$, then $\zeta' =  (\zeta \setminus \mu^*(E_{\zeta'}) )\cup  E_{\zeta'}$
by construction.  Since $E_{\zeta'}\subseteq \mmE^+$ and $\mu^*(E_{\zeta'})\subseteq \mmE^-$, we have
  $$\zeta'= (\zeta\cap \mmE^+) \cup E,\qquad \textrm{ with }\quad E=(\zeta'\cap \mmE^-) \cup E_{\zeta'}.$$ 
This shows the inclusion $\subseteq$ by noting that  $E_{\zeta'}\subseteq \im \mu$ and $\zeta'\cap \mmE^- \subseteq \zeta \cap \mmE^-$.
 
 To prove the  inclusion $\supseteq$ we proceed as follows. By Lemma~\ref{lemma:decompPrePsiEdges}, the set on the right consists of elements in $\Theta_\mmG(B)$.  For $\zeta'= (\zeta\cap \mmE^+) \cup E$,  we have
$E_{\zeta'}=E \cap \mmE^+$  and by the computations above, $\psi(\zeta')=\zeta$.
\end{proof}

Let $(\mmG,\mu)$ be an edge partition. Using  \eqref{eq:decompTheta} and Lemma \ref{lemma:decompPrePsi}, we conclude that  for $i\in\mmN$ it holds that
\begin{align}\label{eq:positive}
\Upsilon_\mmG(i)&=\!\!\!\sum\limits_{\zeta\in\Lambda_\mmG(i)} \ \sum\limits_{\zeta'\in\psi^{-1}(\zeta)}\!\!\!\pi(\zeta')= \!\!\!\sum\limits_{\zeta\in\Lambda_\mmG(i)}\!\!\!\pi(\zeta\cap \mmE^+)\prod\limits_{e\in\zeta\cap \mmE^-}\left(\pi(e)+\sum\limits_{e'\in \mu(e)}\pi(e')\right). 
\end{align}

We are now in position to prove the main result of the section. 

\begin{theorem}\label{thm:posol}
Let $A\in R^{m\times m}$ with $\det(A)\not=0$,  $b\in R^m$, $x=(x_1,\dots,x_m)$, and  $L$ be as in \eqref{eq:Laplacian}. 
If there exists a P-graph with Laplacian $L$, then each  component of the solution to the  linear system $Ax+b=0$ is a quotient of two terms in $\Rnng$.
\end{theorem}
\begin{proof}
Let $\mmG$ be a P-graph with Laplacian $L$ and $\mu$ an associated map.  Using \eqref{eq:positive} and  Definition \ref{def:Pgraph}(\hyperref[cond4Pgraph]{iv}), $\Upsilon_\mmG(i)\in\Rnng$   for all  $i\in\mmN$, since it is a sum of nonnegative terms. The result follows from Proposition \ref{prop:solCMT}.
\end{proof}

 Using  \eqref{eq:positive} and $\pi(\zeta\cap \mmE^+)>0$ for $\zeta\in\Lambda_\mmG(i)$, we might further characterize when 
$\Upsilon_\mmG(i)$ is different from zero, and, in particular, when 
$\det(A)=(-1)^m\Upsilon_\mmG(m+1)$ is different from zero. 
\begin{proposition}\label{prop:possolgraph}
Let $A\in R^{m\times m}$,  $b\in R^m$, $x=(x_1,\dots,x_m)$, and  $L$ be as in \eqref{eq:Laplacian}. 
Assume there exists a P-graph $\mathcal{G}$ with Laplacian $L$ and an associated map $\mu$. For  $i\in  \mmN$, $\Upsilon_\mmG(i)\neq 0$ if and only if 
there exists a spanning tree $\tau\in\Lambda_\mmG(i)$  such that 
$$\pi(e)+\sum\limits_{e'\in\mu(e)}\pi(e')\,\,\in\,\, R_{>0},\qquad \textrm{for all }\quad e \in \tau\cap \mmE^-.$$
In particular, $\det(A)\neq 0$ if and only if the statement holds for $i=m+1$. 
\end{proposition}

The proposition provides a graphical way to check for zero solutions as well, since $\Upsilon_\mmG(i)\neq 0$ if and only if 
  the solution to the  system $Ax+b=0$ satisfies  $x_i\neq 0$.
The proposition holds for any P-graph and any associated map, and thus holds for either all possible P-graphs and associated maps or none.
Further, for $i=m+1$, the vector $b$ plays no role, and hence in order to check whether a square matrix $A$ has nonzero determinant, one can apply the proposition with arbitrary $b$, for example $b=0$.

\begin{remark}
A known criterion for nonnegativity of the solution is the following. If $-A$ is an $M$-matrix and the entries of $b$ are nonnegative, then the solution is nonnegative, because the inverse of $-A$ has nonnegative entries \cite[Section 9.5]{hogben}. Example \ref{ex:split2} is not an $M$-matrix (and cannot be made one by reordering of columns or rows, see also Remark \ref{rem:order}). Oppositely, for $Ax+b=0$ with 
$$A=\begin{pmatrix} -4 & \phantom{-}2 \\  \phantom{-}1& -1\end{pmatrix},\quad b=\begin{pmatrix} 1 \\ 1 \end{pmatrix},$$
$-A$ is an $M$-matrix, but there is not a P-graph in this case as Definition \ref{def:Pgraph}(iv) fails. Hence the two criteria are complementary. 
\end{remark}

We conclude this section with two observations on the existence of P-graphs.
 
\begin{remark}\label{rem:order}
 In many applications the rows of the matrix $A$ (and the vector $b$) have a natural order that corresponds to the order of the variables $x_1,\ldots,x_m$. This is for example the case if $Ax+b=0$ is the equilibrium equations of a linear dynamical system $\dot{x}=Ax+b$.
 Positivity of a solution to $Ax+b=0$ is independent of the order of the rows, however, the existence of a  P-graph is not. 
As pointed out in Remark \ref{rmk:condP}, there cannot exist a P-graph with Laplacian $L$ unless the  diagonal entries of $L$ are nonpositive. This property will generally not be fulfilled if the rows are reordered (without reordering the variables accordingly), as the following example  shows with $m=2$ for two orders of the equations:
 \[
L= \left(\begin{array}{cc|c}
 -2 &  \phantom{-}1 &  \phantom{-}2 \\
 \phantom{-}1 & -2 &  \phantom{-}2 \\ \hline
  \phantom{-}1 &  \phantom{-}1 & -4
 \end{array}\right), \qquad 
 \widehat L= \left(\begin{array}{cc|c}
 -1 &  \phantom{-}2 & -2 \\
  \phantom{-}2 & -1 & -2 \\ \hline
 -1 & -1 &  \phantom{-}4
 \end{array}\right).
 \]
 The canonical multidigraph of the first example is a P-graph as the nondiagonal elements are nonnegative. The second example is obtained from the first by swapping the first two rows, followed by a multiplication with minus one to obtain nonpositive diagonal elements (Remark \ref{rmk:condP}). However this implies that the third, or $(m+1)$-th,  row also changes sign, causing a positive element in the diagonal. Hence there is not a P-graph with Laplacian $\widehat L$.

In general, a reordering of the equations and/or the variables might be convenient in order to find a P-graph corresponding to the system.
 \end{remark}

\begin{remark}\label{rk:split-merge}
Consider a P-graph $\mmG$ with associated map $\mu$ and Laplacian $L$. By splitting edges with positive labels or merging edges with negative labels, the resulting multidigraph is still a P-graph with Laplacian $L$.  

Specifically, if $\mmG$ has several parallel edges $e_1,\dots,e_\ell\in \mmE_{ji}$ with negative labels, the multidigraph $\widehat{\mmG}$ with exactly one edge $e$ from $j$ to $i$ and label $\pi(e_1)+\dots+\pi(e_\ell)$ is also a P-graph. An associated map $\widehat{\mu}$ agreeing with $\mu$ for all edges different from $e$ can be  defined as $\widehat{\mu}(e)= \cup_{i=1}^\ell \mu(e_i)$.  Here we use that $\mu(e_i)$ is also a subset of edges in $\widehat{\mmG}$. The proof is straightforward.

Similarly,   for any edge $e'\in \mmE_{ji}$ with positive label, the multidigraph 
$\widehat{\mmG}$ with $\ell$ edges  $e'_1,\dots,e'_\ell$ from $j$ to $i$ fulfilling $\pi(e'_1)+\dots + \pi(e'_\ell)=\pi(e')$ is also a P-graph with the same Laplacian. An associated map $\widehat{\mu}$ can be defined as  $\widehat{\mu}(e)=(\mu(e)\setminus \{e'\}) \cup \{e'_1,\dots,e'_\ell\}$ if $e'\in \mu(e)$ and $\widehat{\mu}(e)=\mu(e)$ otherwise, with the natural identification of edges in $\mmG$ and $\widehat{\mmG}$.
\end{remark}

\begin{lemma}
Assume $R$ is totally ordered. Let $\mmG=(\mmN,\mmE)$ be a P-graph with Laplacian $L$ such that there are two nodes $i,j\in\mmN$ with $\mmE_{ji}\not=\emptyset$ and $L_{ij}=0$, that is, the labels of the parallel edges from $j$ to $i$ sum to zero. Then there is a P-graph $\mmG'=(\mmN,\mmE')$ with Laplacian $L$ and $\mmE'_{ji}=\emptyset$. The edges of $\mmG'$ and $\mmG$ agree after potentially  splitting some edges with source $j$ into several parallel edges.
\end{lemma}

\begin{proof} 
Let $\mu$ be a map associated with the P-graph $\mmG$.
Without loss of generality, we can assume that there is only one edge $e^-$ from $j$ to $i$ with negative label (cf. Remark~\ref{rk:split-merge}). 
Consider the multidigraph $\widehat{\mmG}=(\mmN, \widehat{\mmE})$ obtained from $\mmG$ by removing the edges in $\mmE_{ji}$ (that is, $\widehat{\mmE}=\mmE\setminus \mmE_{ji}$) and let $\widehat{\mu}$ be the restriction of $\mu$ to $\mmE$:
\[\widehat{\mu}(e) = \mu(e) \cap  \widehat{\mmE}^+=  \mu(e) \cap  (\mmE\setminus \mmE_{ji})^+ ,\qquad e\in \widehat{\mmE}^-.\]
The pair $(\widehat{\mmG}, \widehat{\mu})$ is an edge partition. Let 
  \begin{align} 
  E & =\{e\in\widehat{\mmE}^-\st \pi(e)+\sum_{e'\in\widehat{\mu}(e)}\pi(e')\in R_{\leq 0}\}  \nonumber \\ & \subseteq 
  \{e\in\widehat{\mmE}^-\st s(e)=j \text{ and }\mu(e)\cap \mmE_{ji}\neq \emptyset\}, \label{eq:E}
  \end{align}
where the inclusion is a consequence of the fact that $\mmG$ and $\widehat{\mmG}$ only differ in edges with source $j$ and that $\mmG$ is a P-graph.  We have
\begin{multline*}
\sum_{e\in E\cup \mmE_{ji}^-}\left(\pi(e)+\sum_{e'\in\mu(e)}\pi(e') \right) = 
\sum_{e\in E}\left(\pi(e)+\sum_{e'\in\widehat{\mu}(e)}\pi(e') \right) + \sum_{\substack{e'\in\mu(e^-)\\ e'\notin \mmE_{ji} }}\pi(e')   \\
+\sum_{e\in E} \sum_{e'\in\mu(e)\cap \mmE_{ji} }\pi(e')  
+ \left(\pi(e^-)+\sum_{e'\in\mu(e^-)\cap \mmE_{ji} }\pi(e') \right).
\end{multline*}
This whole sum is in $R_{\geq 0}$ since $\mmG$ is a P-graph with associated map $\mu$, while the summand of the second row   is in $R_{\leq 0}$  because the sum is over labels of edges in $\mmE_{ji}$ and $L_{ij}=\sum_{e\in\mmE_{ji}}\pi(e)=0$. 
We deduce that for $E^\star=\{e'\in \mu(e^-)\st e'\notin \mmE_{ji}\}$, 
\begin{equation}\label{eq:positsumlemma}
\sum_{e\in E}\left(\pi(e)+\sum_{e'\in\widehat{\mu}(e)}\pi(e')\right)+  \sum_{e'\in E^\star}\pi(e') \in R_{\geq 0}.
\end{equation}
The edges in $E^\star$ belong to $\widehat{\mmG}$, but not to $\im \widehat{\mu}$. Roughly speaking, in view of \eqref{eq:positsumlemma}, we will add some of these edges to $\im \widehat{\mu}$ and modify $\widehat{\mmE}$ by splitting some edges.

Let $e\in E$.  Fix an order of the edges in the set $E^\star$, such that $E^\star=\{e'_1,\dots,e'_\ell\}$ and let for $k\in \{1,\dots,\ell\}$,
\[\alpha_k= \left(\pi(e)+\sum_{e'\in\widehat{\mu}(e)}\pi(e')\right)+  \sum_{i=1}^k\pi(e'_i).\]
Choose the first index $k$ such that $\alpha_{k}\geq 0$. 
Let $\beta_1> 0$ and $\beta_2\geq 0$ such that $\alpha_{k-1} + \beta_1=0$ and $\beta_1+\beta_2 = \pi(e'_{k}).$
If $\beta_2>0$, then redefine the multidigraph $\widehat{\mmG}$ by splitting the edge $e'_k$  into two edges $\bar{e}_1,\bar{e}_2$ with labels $\beta_1,\beta_2$  respectively. If $\beta_2=0$, let $\bar{e}_1=e_k'$.
Recall that $E^\star\cap \im \widehat{\mu}=\emptyset$.
Consider the map $\widehat{\mu}'$ given by
$$\widehat{\mu}'(\bar{e})= \begin{cases}
\widehat{\mu}(\bar{e}) & \text{if }\bar{e}\neq e \\ \widehat{\mu}(\bar{e})\cup \{e_1',\dots,e_{k-1}',\bar{e}_1\} & \text{if }\bar{e}= e.
\end{cases}
$$
Then $\widehat{\mu}'$ fulfils 
\begin{equation*}\label{eq:pie}
\pi(\bar{e})+\sum_{e'\in\widehat{\mu}'(\bar{e})}\pi(e') \quad  \begin{cases} =\pi(\bar{e})+\sum_{e'\in\widehat{\mu}(\bar{e})}\pi(e') \in  R_{\geq 0} & \text{if }\bar{e}\notin E \\
= 0 & \text{if }\bar{e}= e \\
\notin R_{\geq 0} & \text{if }\bar{e}\in E\setminus \{e\}.
\end{cases}
\end{equation*}
Further, \eqref{eq:positsumlemma} holds with $E$, $\widehat{\mu}$ and $E^\star$ replaced by $E\setminus \{e\}$, $\widehat{\mu}'$ and $E^\star\setminus \{e_1',\dots,e_{k-1}',\bar{e}_1\}$, respectively.

By iterating this construction for all edges in $E$, we obtain a multidigraph $\widehat{\mmG}$ with Laplacian $L$ and a map $\widehat{\mu}'$.  $\widehat{\mmG}$ differs from $\mmG$ in that $\mmE_{ji}=\emptyset$ and the edges in $E^\star$ might have been split. In particular both multidigraphs agree on edges with sources different from $j$. The map $\widehat{\mu}'$ fulfils Definition \ref{def:Pgraph}(\hyperref[cond4Pgraph]{iv}) by construction. 

All that remains is to show that $(\widehat{\mmG},\widehat{\mu}')$ is an edge partition. Conditions (\ref{cond3aPgraph}) and (\ref{cond3cPgraph}) are readily  satisfied. By Remark~\ref{rk:split-merge}, the pair $(\widehat{\mmG},\widehat{\mu}')$ fulfils (\ref{cond3bPgraph}) for all edges other than the edges $e'\in \mu(e^-)\cap \widehat{\mu}'(e)$ with $e\in E$. We only need to prove (\ref{cond3bPgraph}) for these edges.
Consider a path from $t(e')$ to $j$  in $\widehat{\mmG}$. We need to show that it contains $t(e)$. 
Since there is no edge with source $j$, the path is also in $\mmG$. Since $\mmG$ is a P-graph with associated map $\mu$ and $t(e^-)=i$, the path contains $i$. By considering any edge in  $\mu(e)\cap \mmE_{ji}\neq \emptyset$ 
(cf.\,\eqref{eq:E}), the 
 subpath from $i$ to $j$ contains $t(e)$. So condition  (\ref{cond3bPgraph}) is satisfied. 
This concludes the proof.
\end{proof}

The lemma has the consequence that to construct a P-graph corresponding to a Laplacian $L$,  we do not need  to consider edges between nodes with zero entry in $L$.

\section{An extension of the previous statements}\label{sec:secondsystem} 
In this section we consider a generalization of the system studied in Section \ref{sec:possolution}.
The system of interest is a linear square system $Ax+b=0$ in $x=(x_1,\dots, x_m)$ such that the coefficient matrix $A$ and the vector of independent terms $b$ are of the form
\begin{equation}\label{eq:secondsystem}
A=\left(
\begin{array}{ccccc}
A_1 & 0 &\cdots& 0 & 0 \\
0  & A_2  & \cdots & 0 & 0  \\
\vdots & \vdots & \ddots & \vdots & \vdots \\
0  &0& \cdots & A_d & 0 \\ \hline
\multicolumn{5}{c}{A_0}  \end{array}
\right)\in R^{m\times m}, \qquad 
b= \left(\begin{array}{c}
b^1    \\
b^2\\ 
\vdots\\
 b^d  \\ \hline
b^0  \end{array}\right)\in R^m,
\end{equation}
with $A_0\in R^{m_0\times m}$ and $b^0\in R^{m_0}$ arbitrary, and for  $i=1,\dots,d$,
\begin{enumerate}[(i)]
\item $A_i$ is a square matrix of size $m_i$.
\item  $b^i$ is a vector of size $m_i$ and nonzero in at most one entry.  
\end{enumerate}
We let $\mmN=\{1,\ldots,m+1\}$ as before and let $\mmN_i$ denote the set of indices of the rows corresponding to $A_i$. Specifically,
\begin{align*}
 \mmN_i &= \left\{1+ \sum_{j=1}^{i-1} m_j\,,\, \dots\, , \, \sum_{j=1}^{i} m_j\right\}, \quad i=1,\dots,d 
 \end{align*}
 Let also
 \begin{align*}
  \mmN_0 &= \mmN\setminus\bigcup\limits_{i=1}^d \mmN_i = \{m-m_0+1,\dots,m+1\}. 
  \end{align*}
Note that we have $m-m_0=m_1+\dots+m_d$.
By (ii), we can choose indices $j_1,\dots,j_d$, with $j_i\in \mmN_i$ for all $i=1,\dots,d$, such that $b_j=0$ 
 if $j\neq j_i$ and $j\leq m-m_0$. If  $b^i$ is the zero vector, then the index $j_i$ is arbitrary (but fixed). Otherwise it is uniquely determined.

 If $d=0$, then we are left with a system of the form studied in Section \ref{sec:possolution}. Hence \eqref{eq:secondsystem} might be considered an extension of the previous case.  The linear system has $d+1$ ``blocks'' or subsystems. For the  variables with indices in $\mmN_i$, $i=1,\ldots,d$, there is one subsystem with $|\mmN_i|$ linear equations. In addition, there is the subsystem  $A_0x+b^0=0$ that might depend on all $m$ variables.

\begin{remark}
If $\det(A)\not=0$, then we might in principle apply the results of the previous section. However, if the rows $j_1,\dots,j_d$ of $A$  are nonnegative and $b^1,\dots,b^{d}$ nonpositive,  then the conditions of Definition \ref{def:Pgraph} will often fail for the examples we have in mind, see  Example \ref{ex:secondtype}. In reaction network theory, which is our main source of examples,  this type of system  arises naturally (after reordering of the equations) and is perhaps the rule rather than the exception.  The equations with index different from $j_i$ correspond to equilibrium equations, and those with index equal to one of $j_i$ correspond to conservation relations with $b_{j_i}<0$.  
\end{remark}

\refstepcounter{theorem}\label{ref:2type}
\newcounter{mycounterType}
\renewcommand{\themycounterType}{\getrefnumber{ref:2type}\,(part\,\Alph{mycounterType})}
\newtheorem{myexample2Type}[mycounterType]{Example}

\begin{myexample2Type}\label{ex:secondtype}
Let $z_1,\dots, z_5$ be the coordinate functions in $\R^5_{\geq 0}$ and consider the linear system in $x_1,x_2,x_3$,
\begin{align*}
-z_2x_1+z_3x_2&=0,\\
x_1+x_2-z_1&=0,\\
z_3x_2-z_4x_3+z_5&=0.
\end{align*}
Note that $\det(A)=(z_2+z_3)z_4\neq 0$.  To apply Theorem \ref{thm:posol}, we change the sign of the second equation, such that the Laplacian matrix $L$ associated with the system has nonpositive diagonal entries, which is a necessary condition for the existence of a P-graph (see Remark \ref{rmk:condP}). 
 The matrix $L$ and a multidigraph with Laplacian $L$ are

\begin{minipage}{0.52\linewidth}
\[
L=\left(\begin{array}{ccc|c}
-z_2 & z_3 & 0 &0 \\ 
-1 & -1 & 0 &  z_1\\ 
0 & \phantom{-}z_3 & -z_4 &z_5\\
\hline
1+z_2 & 1-2z_3 & z_4 & -z_1-z_{5}
\end{array} \right)
\]
\end{minipage}
\begin{minipage}{0.4\linewidth}
\begin{center}
\begin{tikzpicture}[inner sep=1.2pt]
\node (v3) [shape=circle] at (4,0) {$3$};
\node (v1) [shape=circle] at (0,0) {$1$};
\node (v2) [shape=circle] at (2,0) {$2$};
\node (*) [shape=circle] at (2,-1.5) {$4$};

\draw[->] (v1) to[out=10,in=170] node[above,sloped]{\footnotesize $-1$}(v2);
\draw[->] (v2) to[out=190,in=-10] node[below,sloped]{\footnotesize $z_3$}(v1);
\draw[->] (v1) to[out=290,in=160] node[left]{\footnotesize $1+z_2$}(*);
\draw[->] (*) to[out=80,in=280] node[right]{\footnotesize $z_1$}(v2);
\draw[->] (v2) to[out=260,in=100] node[left]{\footnotesize $1$}(*);
\draw[->] (v2) to[out=230,in=130] node[left]{\footnotesize $-2z_3$}(*);
\draw[->] (v2) to[out=0,in=180] node[above]{\footnotesize $z_3$}(v3);
\draw[->] (v3) to[out=260,in=30] node[above]{\footnotesize $z_4$}(*);
\draw[->] (*) to[out=10,in=280] node[below]{\footnotesize $z_5$}(v3);
\end{tikzpicture}
\end{center}
\end{minipage}

\medskip
 Since $\mu(2\ce{->[-2z_3]}4)$ is necessarily a subset  of $\{2\ce{->[z_3]}3, 2\ce{->[1]}4\}$, Defintion  \ref{def:Pgraph}(\hyperref[cond4Pgraph]{iv}) cannot be satisfied. Hence this multidigraph is not a P-graph.  Any multidigraph with Laplacian $L$ will have the same problem, so Theorem \ref{thm:posol} cannot be applied.
If we substitute $z_3$ by a nonnegative real number $\le1$, then the multidigraph is a P-graph.

The linear system falls in the setting of the present section with 
$$ A_1= \left(\begin{array}{cc}
-z_2 & z_3  \\ 
1 & 1  \end{array} \right), \quad A_0 = \left(\begin{array}{ccc}
0 & \phantom{-}z_3 & -z_4 
\end{array} \right), \quad b^1 = \left(\begin{array}{c}
0 \\ 
  -z_1
\end{array} \right),\quad b^0 =  \left(\begin{array}{c}
  z_5\end{array} \right).$$
\end{myexample2Type}

 \begin{definition}\label{def:Acompatible}
Let  $\mmG$  be a labeled multidigraph with $m+1$ nodes and Laplacian $L$. Then $\mmG$  is said to be \textbf{$\boldsymbol{A}$-compatible} if
 \begin{enumerate}[(i)]
\item  There is not an edge from a node in $\mmN_i$, $i\ge 0$,  to a node in $\mmN_j$ for $i\neq j$, $j\geq 1$.
\item The $\ell$-th row of $L$ agrees with the $\ell$-th row of $A|b$ for $\ell\not\in\{j_1,\ldots,j_d,m+1\}$. 
\end{enumerate}
Furthermore, the Laplacian $L$ is said to be $A$-compatible, if $\mathcal{G}$ is $A$-compatible.
\end{definition}

The graphical structure of an $A$-compatible multidigraph  is shown in Figure \ref{fig:struct}. By Definition~\ref{def:Acompatible}(i), any $A$-compatible Laplacian has the same block form as $A|b$ in \eqref{eq:secondsystem}. In particular, the support of the $j_i$-th row of $L$ is included in $\mmN_i\cup \{m+1\}$.

\begin{figure}[h!]
\begin{center}
\begin{tikzpicture}[inner sep=1.2pt]
\node (Xj) [shape=circle,draw,minimum size=1.2cm] at (2,-0.5) {$\mmN_j$};
\node (...1) at (1.4,0.9) {$\ddots$};
\node (X1) [shape=circle,draw,minimum size=1.2cm] at (1,2) {$\mmN_{1}$};
\node (..2) at (4,-0.7) {$\dots$};
\node (Xd) [shape=circle,draw,minimum size=1.2cm] at (7,2) {$\mmN_d$};
\node (...3) at (6.5,0.9) {\reflectbox{$\ddots$}};
\node (Xk) [shape=circle,draw,minimum size=1.2cm] at (6,-0.5) {$\mmN_k$};
\node (X0) [shape=circle,draw,minimum size=1.2cm] at (4,2) {$\mmN'_0$};
\node (*) at (4,0) {$m+1$};
\draw[->] (X1) to[out=10,in=170] node[above,sloped]{}(X0);
\draw[->] (X1) to[out=-10,in=180] node[below,sloped]{}(*);
\draw[->] (Xd) to[out=180,in=-10] node[above,sloped]{}(X0);
\draw[->] (Xk) to[out=135,in=0] node[right]{}(*);
\draw[->] (X0) to[out=250,in=110] node[right]{}(*);
\draw[->] (*) to[out=60,in=300] node[right]{}(X0);
\end{tikzpicture}
\end{center}
\caption{The structure of an $A$-compatible multidigraph $\mmG$ with the node $m+1$ singled out, $\mmN_0'=\mmN_0\setminus\{m+1\}$.}\label{fig:struct}
\end{figure}
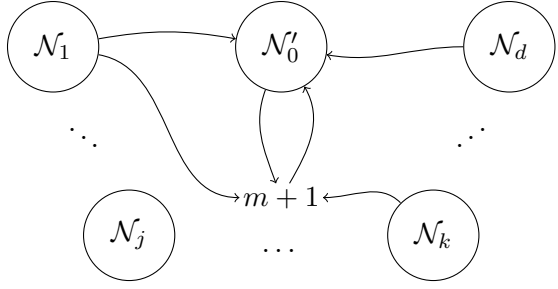

\begin{remark}
If $d=0$, then any multidigraph with Laplacian as in \eqref{eq:Laplacian} is $A$-compatible. Such a  multidigraph has only  the component with node set $\mmN_0$.
\end{remark}

Let $F= \{j_1,\dots,j_d,m+1\}$ and $\mmN_{d+1}=\{m+1\},$
and define 
\begin{equation}\label{eq:B}
\B=\underset{j=1}{\overset{d+1}{\mathlarger{\odot}}} \mmN_j,\qquad \B^k=\underset{j=1,j\neq k}{\overset{d+1}{\mathlarger{\odot}}} \mmN_j,\quad k=1,\ldots,d+1. 
\end{equation}
If $B$ is an element in any of the defined sets, then the set $B\cap \mmN_i$ ($i\neq k$ in the second case) consists of a single element, which we  for simplicity denote by $\beta_i$, that is,
\begin{equation}\label{eq:beta}
B\cap \mmN_i=\{\beta_i\}.
\end{equation}

 It is easy  to show the following lemma using  Definition~\ref{def:Acompatible}(i)  (see Figure~\ref{fig:struct}).

\begin{lemma}\label{lemma:rootsforests}
Let $\mmG$ be an $A$-compatible multidigraph.
\begin{enumerate}[(i)]
\item Let $\zeta$  a spanning forest of $\mmG$ and $\tau$ a connected component of $\zeta$. If $\tau$ contains a node in $\mmN_i$, $i\ge 0$, then the  root of $\tau$ is in  $\mmN_i\cup \mmN_0$. 
\item  $\Theta_{\mmG}(F,B)=\emptyset$  if $B$ contains two elements in $\mmN_i$ for some $i>0$.
\end{enumerate}
\end{lemma}

In particular, if $\tau$ contains $m+1$, then the root of $\tau$ is in $\mmN_0$.

We are now ready to state a parallel version of Proposition \ref{prop:solCMT} under the assumptions of  the current setting. 
The proof is given in Section \ref{sec:computation}.

\begin{proposition}\label{prop:solCMT2}
Consider a linear system  $Ax+b=0$ as in  \eqref{eq:secondsystem} such that $\det(A)\neq 0$,  and assume there exists an $A$-compatible multidigraph $\mmG$.  Then, the solution to the linear system 
is 
\begin{align}\label{eq:solution2}
x_\ell &=\frac{\sum\limits_{k=1}^{d+1} (-b_{j_k}) \sum\limits_{B\in \B^{k},\ell\notin B}\left(\prod\limits_{ i=1, i\neq k}^d a_{j_i\beta_i}\right) \Upsilon_{\mmG}(F,B\cup \{\ell\})   }{\sum\limits_{B\in \B}\left(\prod\limits_{i=1}^d a_{j_i\beta_i}\right)\Upsilon_{\mmG}(F,B)},
\end{align}
where $b_{j_{d+1}}=-1$ for convenience and $\ell=1,\ldots,m$.
\end{proposition}

By assuming $d=0$ we retrieve Proposition \ref{prop:solCMT}. 
Several terms in the numerator of \eqref{eq:solution2} are readily seen to be zero. Indeed, by Lemma~\ref{lemma:rootsforests}(ii), $\Theta_{\mmG}(F,B\cup \{\ell\})=\emptyset$ if $\ell \in \mmN_i$ for $i\in \{1,\dots,d\}$ and $B\in \B^k$ with  $k\in \{1,\dots,d+1\}$, $k\neq i$.

If the columns of $A_0$ with indices in $\mmN_k$, $k\in \{1,\dots,d\}$, are zero, then
the variables $x_i$ with $i\in \mmN_k$ only appear in the subsystem given by the rows of $A$ with indices in $\mmN_k$ and thus this subsystem can be solved independently. 

Building on the ideas of Section~\ref{sec:possolution}, Proposition~\ref{prop:solCMT2} allows us to study when the solution to the system $Ax+b=0$ is positive. In particular, we give the following characterization. The proof is given in Section \ref{sec:positivity}.

\begin{theorem}\label{thm:possol2}
Consider a linear system as in  \eqref{eq:secondsystem} such that $\det(A)\neq 0$, the rows $j_1,\dots,j_d$ of $A$ are nonnegative and $b_{j_1},\dots,b_{j_d}$ are nonpositive. Further, assume there exists an $A$-compatible P-graph $\mmG$ such that
\begin{itemize}
\item[(*)]\label{condpossol2}  for $\ell\in\{1,\dots,m\}$ and $i\in\{1,\dots, d\}$, any path from $j_i\in\mmN_i$ to $\ell$ that contains an edge in $\mmE^-$  goes through $m+1$.
\end{itemize}
Then, each component of the solution in \eqref{eq:solution2} is the quotient of two terms in  $R_{\geq 0}$.
\end{theorem}

In particular, if the target of all edges with negative labels is $m+1$, then condition (\hyperref[condpossol2]{*}) is fulfilled, see Example \ref{ex:CRNT} for an illustration. If $d=0$, then any $A$-compatible multidigraph $\mmG$ satisfies (\hyperref[condpossol2]{*}), so Theorem \ref{thm:possol2} is a generalization of Theorem \ref{thm:posol}.

\begin{corollary}\label{cor:extrazeros}
With the hypotheses of Theorem \ref{thm:possol2}, assume there is a path from a node $\ell\in\mmN_i$ with $i>0$, to a node   $j\in\{1,\dots,m\}$ that contains an edge in $\mmE^-$ and that does not go through $m+1$. Then, the solution to the linear system $Ax+b=0$ fulfils  $x_\ell=0$.
\end{corollary}

\begin{remark}
Remark \ref{rem:order} could essentially be restated here. There is some flexibility to choose the precise order of the rows of $A$ within each block. All the rows of $A_1,\dots,A_d$  can be reordered indiscriminately as well as those of $A_0$ (and $b^0$).  This would lead to different multidigraphs.
\end{remark}

\begin{myexample2Type}\label{ex:secondtype2}
Consider the linear system in Example \ref{ex:secondtype}. 
We have $\mmN_1=\{1,2\}$ and $\mmN_0=\{3,4\}$.
The following multidigraph is $A$-compatible 
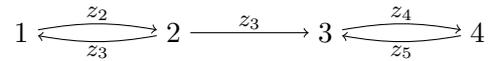
\begin{center}
\begin{tikzpicture}[inner sep=1.2pt]
\node (v3) [shape=circle] at (4,0) {$3$};
\node (v1) [shape=circle] at (0,0) {$1$};
\node (v2) [shape=circle] at (2,0) {$2$};
\node (*) [shape=circle] at (6,0) {$4$};
\draw[->] (v1) to[out=10,in=170] node[above,sloped]{\footnotesize $z_2$}(v2);
\draw[->] (v2) to[out=190,in=-10] node[below,sloped]{\footnotesize $z_3$}(v1);
\draw[->] (v2) to[out=0,in=180] node[above]{\footnotesize $z_3$}(v3);
\draw[->] (v3) to[out=10,in=170] node[above]{\footnotesize $z_4$}(*);
\draw[->] (*) to[out=190,in=-10] node[below]{\footnotesize $z_5$}(v3);
\end{tikzpicture}
\end{center}
Indeed, its Laplacian $L$ is $A$-compatible:
$$
L=\left(\begin{array}{rrrr} -z_2 & z_3 & 0&0\\z_2 & -2z_3 & 0 & 0\\ 0 & z_3 & -z_4&z_5\\ 0 & 0 & z_4 & -z_5 \end{array} \right). 
$$
The multidigraph is further a P-graph that satisfies (\hyperref[condpossol2]{*}) in Theorem~\ref{thm:possol2}, since there are no edges with negative label.
Thus, we conclude by Theorem \ref{thm:possol2} that the solution to the linear system is nonnegative.
\end{myexample2Type}

\begin{example} 
Consider the following linear system in five variables $x_1,\dots, x_5$ and assume $z_1,\dots,z_9\in \Rnng$.
\[
\left(\begin{array}{ccccc}
-z_1& z_2 & 0 & 0 & 0\\ 
1 & 1 & 0 & 0 & 0\\
0&z_2&-z_3-z_4& z_5 & 0\\ 
0&0& z_3&-z_5&0\\ 
0&0&z_3&z_5&-z_6
\end{array}\right)
\left(\begin{array}{c} x_1 \\ x_2\\ x_3\\ x_4\\ x_5 \end{array}\right) +
\left(\begin{array}{c} 0 \\ -z_7\\ z_8\\ 0\\ z_9  \end{array}\right)=  
\left(\begin{array}{c} 0 \\ 0\\ 0\\ 0\\ 0\end{array}\right).
\] 
This system has the form of \eqref{eq:secondsystem} with $d=1$, $j_1=2$, $m_1=2$ and $m_0=3$.  The following is an $A$-compatible Laplacian 
\[
\left(\begin{array}{cccccc}
-z_1& z_2 & 0 & 0 & 0 & 0\\ 
\boldsymbol{z_1} & \boldsymbol{-2z_2} & \boldsymbol{0} & \boldsymbol{0} & \boldsymbol{0} & \boldsymbol{0}\\
0&z_2&-z_3-z_4& z_5 & 0 & z_8\\ 
0&0& z_3&-z_5&0 & 0\\ 
0&0&z_3&z_5&-z_6 & z_9\\
\boldsymbol{0}& \boldsymbol{0} &\boldsymbol{z_4-z_3}& \boldsymbol{-z_5} &\boldsymbol{z_6}& \boldsymbol{-z_8-z_9}
\end{array}\right).
\]
The $2$nd row of $A|b$ is replaced by another vector with the same support and such that the number of negative entries outside the diagonal are kept as small as possible.   

An example of a P-graph with Laplacian $L$ and associated map $\mu$ is:

\noindent \begin{minipage}{0.69\linewidth}
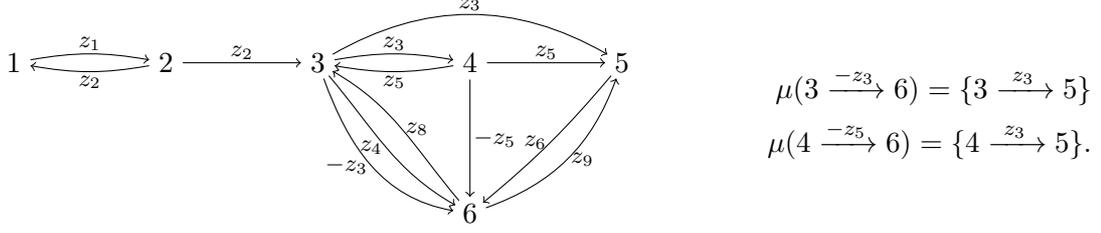
\begin{center}
\begin{tikzpicture}[inner sep=1.2pt]
\node (v1) [shape=circle] at (0,0) {$1$};
\node (v2) [shape=circle] at (2,0) {$2$};
\node (v3) [shape=circle] at (4,0) {$3$};
\node (v4) [shape=circle] at (6,0) {$4$};
\node (v5) [shape=circle] at (8,0) {$5$};
\node (*) [shape=circle] at (6,-2) {$6$};

\draw[->] (v1) to[out=10,in=170] node[above,sloped]{\footnotesize $z_1$}(v2);
\draw[->] (v2) to[out=190,in=-10] node[below,sloped]{\footnotesize $z_2$}(v1);
\draw[->] (v2) to[out=0,in=180] node[above]{\footnotesize $z_2$}(v3);
\draw[->] (v3) to[out=-70,in=170] node[left]{\footnotesize $-z_3$}(*);
\draw[->] (v3) to[out=-50,in=150] node[left]{\footnotesize $z_4$}(*);
\draw[->] (*) to[out=130,in=-30] node[right]{\footnotesize $z_8$}(v3);
\draw[->] (v3) to[out=10,in=170] node[above]{\footnotesize $z_3$}(v4);
\draw[->] (v4) to[out=190,in=-10] node[below]{\footnotesize $z_5$}(v3);
\draw[->] (v3) to[out=30,in=150] node[above]{\footnotesize $z_3$}(v5);
\draw[->] (v4) to[out=0,in=180] node[above]{\footnotesize $z_5$}(v5);
\draw[->] (v4) to[out=270,in=90] node[right]{\footnotesize $-z_5$}(*);
\draw[->] (v5) to[out=230,in=40] node[left]{\footnotesize $z_6$}(*);
\draw[->] (*) to[out=20,in=250] node[right]{\footnotesize $z_9$}(v5);
\end{tikzpicture}
\end{center}
\end{minipage}
\begin{minipage}{0.3\linewidth}
\begin{align*}
\mu(3\ce{->[-z_3]} 6)=\{3\ce{->[z_3]} 5\}\\
\mu(4\ce{->[-z_5]}6)=\{4\ce{->[z_3]} 5\}.
\end{align*}
\vfill
\end{minipage}

All edges in $\mmE^-$ have target node $6=m+1$. Therefore, by Theorem \ref{thm:possol2} the solution to the linear system is nonnegative. Indeed, the solution is
\begin{align*}
x_{1}=& \frac {z_7z_2}{z_2+z_1},\qquad
x_{2}= \frac {z_1z_7}{z_2+z_1},\qquad
x_{3}= \frac {z_1z_2z_7+(z_1+z_2)z_8}{z_4 \left( z_2+z_1 \right)},\\
x_{4}=& \frac {z_3 \left( z_1z_2z_7+(z_1+z_2)z_8 \right) }{z_4 \left( z_2+z_1 \right) z_5},\qquad
x_{5}= \frac {2\,z_1z_2z_3z_7+(z_1+z_2)(2\,z_3z_8+z_4z_9)}{z_6z_4 \left( z_2+z_1 \right) }.
\end{align*}
\end{example}

\begin{example}\label{ex:CRNT}
We consider the following reaction network  \cite{FSWFS16}.  For convenience we denote the chemical species by $X_1,\ldots,X_6$, and $\k_1,\dots,\k_{11}$ denote (unknown) reaction rate constants, one for each of the 11 reactions:
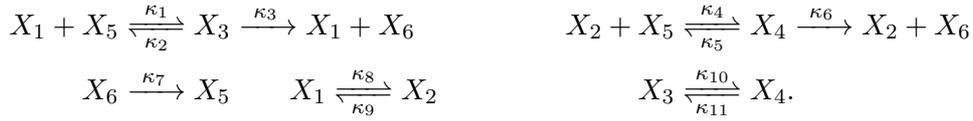
\begin{align*}
X_1 + X_5 &  \cee{<=>[\k_1][\k_2] } X_3 \cee{->[\k_3]} X_1+ X_6  &
X_2 + X_5 &  \cee{<=>[\k_4][\k_5] } X_4 \cee{->[\k_6]} X_2+ X_6\\
X_6 & \cee{->[\k_7]} X_5 \qquad
X_1  \cee{<=>[\k_8][\k_9] } X_2  &  
X_3 & \cee{<=>[\k_{10}][\k_{11}] } X_4.
\end{align*}

Further, let $x_1,\dots,x_6$  denote the concentrations (in some unit) of the corresponding species. Assuming mass-action kinetics, the evolution of the concentrations are described by an ODE system of the form,
\begin{align*}
\dot{x}_1 &= -\k_{1}x_{1}x_{5}+(\k_{2}+\k_{3})x_{3}-\k_{8}x_{1}+\k_{9}x_{2}\\ 
\dot{x}_2 &= -\k_{4}x_{2}x_{5}+(\k_{5}+\k_{6})x_{4}+\k_{8}x_{1}-\k_{9}x_{2} \\ 
\dot{x}_3 &= \phantom{-}\k_{1}x_{1}x_{5}-(\k_{2}+\k_{3})x_{3}-\k_{10}x_{3}+\k_{11}x_{4} \\ 
\dot{x}_4 &= \phantom{-}\k_{4}x_{2}x_{5}-(\k_{5}+\k_{6})x_{4}+\k_{10}x_{3}-\k_{11}x_{4} \\ 
\dot{x}_5 &= -\k_{1}x_{1}x_{5}-\k_{4}x_{2}x_{5}+\k_{2}x_{3}+\k_{5}x_{4}+\k_{7}x_{6} \\ 
\dot{x}_6 &= \phantom{-}\k_{3}x_{3}+\k_{6}x_{4}-\k_{7}x_{6}.
\end{align*}
As is evident from the equations, there are two conserved quantities, ${x}_{1} + {x}_{2}  + {x}_{3} + {x}_{4}=T_1$ and $x_3+x_4+ x_5 + x_6 = T_2$ with $T_1,T_2\in\R_{\geq 0}$. 

Using the theory developed in this section we give a one-dimensional parameterization (in $x_5$) of the positive steady state variety constrained to  the conservation equation given by $T_1$,
$$V_{T_1}=\Big\{x\in\R^6_{>0}\,|\, \dot{x}_i=0, i=1,\ldots,6\Big\}\cap \Big\{x\in\R^6_{>0}\,|\,  {x}_{1} + {x}_{2}  + {x}_{3} + {x}_{4}=T_1\Big\}.$$
Consider the variable $x_5$ as an extra constant of the system. Two equations, for example $\dot{x}_4=0$ and $\dot{x}_5=0$, are redundant at steady state because of the conservation equations, and might be removed. Thus the elements of $V_{T_1}$ fulfil:

{\smaller
\[
\left(\begin{array}{ccccc}
-(\kappa_1x_5+\kappa_8) & \kappa_9 & \kappa_2 & 0 & 0\\ 
\kappa_{8}&-(\kappa_{4}x_{5}+\kappa_9)&0&\kappa_{5}+\kappa_{6}& 0\\ 
\kappa_{1}x_{5}&0& -(\kappa_{2}+\kappa_{3}+\kappa_{10})&\kappa_{11}&0\\ 
1 & 1 & 1 & 1 & 0\\
0&0&\kappa_{3}&\kappa_{6}&-\kappa_{7}
\end{array}\right)
\left(\begin{array}{c} x_1 \\ x_2\\ x_3\\ x_4\\ x_6 \end{array}\right) =
\left(\begin{array}{c} 0 \\ 0\\ 0\\ T_1\\ 0  \end{array}\right).  
\] 
}

This system has the form of \eqref{eq:secondsystem} with $d=1$ and $j_1=4$.  Consider the following $A$-compatible Laplacian 

{\small
\[
\left(\begin{array}{cccccc}
-(\kappa_1x_5+\kappa_8) & \kappa_9 & \kappa_2 & 0 & 0& 0\\ 
\kappa_{8}&-(\kappa_{4}x_{5}+\kappa_9)&0&\kappa_{5}+\kappa_{6}& 0 & 0\\ 
\kappa_{1}x_{5}&0& -(\kappa_{2}+\kappa_{3}+\kappa_{10})&\kappa_{11}&0 & 0\\ 
\boldsymbol{0} & \boldsymbol{\kappa_{4}x_{5}} & \boldsymbol{\kappa_{10}} & \boldsymbol{-(\kappa_5+2\kappa_6+\kappa_{11})} & \boldsymbol{0} & \boldsymbol{0}\\
0&0&\kappa_{3}&\kappa_{6}&-\kappa_{7} & 0\\
\boldsymbol{0} &\boldsymbol{0} & \boldsymbol{0}& \boldsymbol{0} &\boldsymbol{\kappa_7} & \boldsymbol{0} 
\end{array}\right),
\]
}
where the $4$-th row (in bold) differs
from the one in  the matrix of the linear system, and the bottom bold row is the $(m+1)$-th  row.
 
The canonical multidigraph with this Laplacian does not have edges with negative label, hence it is a P-graph and condition (\hyperref[condpossol2]{*})   is fulfilled.  Therefore, by Theorem~\ref{thm:possol2}, the solution to the system is nonnegative:

{\smaller
\begin{align*}
x_{1} =&\frac{T_1}{q(x)} \Big( (\k_2+\k_3)\k_{4}\k_{11}x_{5}+\k_9((\k_2+\k_3)(\k_5+\k_6) +(\k_2+\k_3) \k_{11}+(\k_5+\k_6)\k_{10})\Big),  \\
x_{2} =&\frac{T_1}{q(x)} \Big( (\k_5+\k_6)\k_{1}\k_{10}x_{5}+\k_8((\k_2+\k_3)(\k_5+\k_6) +(\k_2+\k_3) \k_{11}+(\k_5+\k_6)\k_{10})\Big),  \\
x_{3} =&\frac{T_1 x_{5}}{q(x)} \Big(  \k_{1}\k_{4}\k_{11}x_{5}+\k_{1}\k_{9}(\k_5+\k_6+\k_{11})+\k_{4}\k_{8}\k_{11} \Big),  \\
x_{4} =&\frac{T_1 x_{5}}{q(x)} \Big( \k_{1}\k_{4}\k_{10}x_{5}+\k_{4}\k_{8}(\k_2+\k_3+\k_{10})+\k_{1}\k_{9}\k_{10}\Big),  \\
x_{6} =&\frac{ T_1x_5}{\k_7q(x)}  \Big(  \k_{{1}}\k_4\left(\k_{{3}}\k_{{11}}+\k_{{6}}\k_{{10}} \right)x_5+\k_{1}\k_{3}\k_{9} ( \k_{5}+\k_{11} ) +\k_{4}\k_{8}( \k_{2}\k_{6}+\k_{3}\k_{11}) + \\
&\k_{6} ( \k_{3}+\k_{10} ) ( \k_{1}\k_{9}+\k_{4}\k_{8}) \Big), 
\end{align*}
}
where $q(x)$ is the following polynomial  in $x_5$,
{\smaller 
\begin{align}
q(x)  =& \k_{1}\k_{4}  ( \k_{10}+\k_{11}  ) x_5^2 \nonumber \\ & + ((\k_2+\k_3)\k_{4}( \k_{8}+\k_{11} ) +(\k_5+\k_6)\k_{1}
 ( \k_{9}+\k_{10}  ) + ( \k_{10}+\k_{11}  ) 
 ( \k_{1}\k_{9}+\k_{4}\k_{8} ) )x_5 \nonumber \\ & + ( \k_{8}+\k_{9})( (\k_2+\k_3)  (\k_5+\k_6+ \k_{11}) +\k_{10}(\k_5+\k_6)  ). \nonumber
 \end{align}
}
This reaction network is one of many reaction networks that fulfill the hypotheses of Theorem \ref{thm:possol2}. In fact, structural conditions on the reaction network guarantee the hypotheses are fulfilled \cite{Saez:elimCRNT}.
\end{example}

We conclude with some remarks about how to find an $A$-compatible multidigraph $\mmG$ fulfilling the requirements of Theorem~\ref{thm:possol2}. 

By Remark~\ref{rmk:condP}, a necessary condition for an $A$-compatible P-graph to exist is that  the diagonal entries of $A$ with indices different from $j_i$, $i=1,\dots, d$, are nonpositive. 
Further, consider  the subsystem $A'_0x'+b^0=0$, where $A'_0$ is the square matrix consisting of the last $m_0$ columns of $A_0$ and $x'=(x_{m-m_0+1},\dots, x_m)$. Then the submultidigraph of  an $A$-compatible P-graph  induced  by the set of nodes $\mmN_0$ is a P-graph with Laplacian $L$ constructed as in \eqref{eq:Laplacian} for $A_0'$ and $b^0$, up to indexing of the nodes.  Thus, if a P-graph for such a subsystem does not exist (see Section \ref{sec:possolution}), then there is not an $A$-compatible P-graph for the original system.

If these necessary conditions for the existence of an $A$-compatible P-graph are fulfilled, we attempt to find an $A$-compatible Laplacian $L$ by minimizing the number of negative entries outside the diagonal. The focus is on the undetermined rows $j_1,\dots,j_d$ of $L$.  Consider $i\in \mmN_k$, $k>0$, and the $i$-th column sum of $A$ without $a_{j_ki}$. 
If this sum is nonpositive for $i\neq j_k$, or nonnegative for $i=j_k$, 
 then a good strategy is to define $L_{j_ki}$ as  minus this sum. This gives nonnegative entries in the $j_k$-th row of $L$ outside the diagonal and a nonpositive entry in the diagonal, while having the $i$-th entry of row $m+1$ equal to zero.

\section{The proofs of Proposition \ref{prop:solCMT2} and Theorem \ref{thm:possol2}}\label{sec:proofs}

\subsection{ Finding the solution to the linear system (Proposition \ref{prop:solCMT2})}\label{sec:computation}
We find expressions for $\det(A)$ and $\det(A_{\ell\rightarrow -b})$ in terms of  the spanning forests of $\mmG$ and the coefficients in the rows $j_1,\dots,j_d$ of $A$ and $b$.
These expressions   are found using Theorem \ref{thm:matrixtree}. The explicit solution to the linear system is subsequently found using Cramer's rule, as in the  proof of Proposition~\ref{prop:solCMT}. This  will give Proposition \ref{prop:solCMT2}.

Recall the definition of the sets $F$, $\B$ and $\B^k$ in \eqref{eq:B} and of $\beta_i$ in \eqref{eq:beta}.
We start with an observation about the form of the spanning forests in $\mmG$.  By applying Lemma~\ref{lemma:rootsforests} repeatedly to a spanning forest $\zeta\in \Theta_{\mmG}(F,B)$ with $B=\widetilde{B}\cup \{\ell\}$,  $\widetilde{B}\in \B^k$, $\ell\notin \widetilde{B}$, for $k\in \{1,\dots,d+1\}$, $\ell \in \{1,\dots,m+1\}$,  we obtain 
\begin{align}\label{eq:rootsforests}
g_\zeta(j_i) & = \begin{cases}
 \beta_i & \textrm{if }k\neq i, \\
 \ell  & \textrm{if }k= i,
\end{cases} & 
g_\zeta(m+1) & = \begin{cases}
m+1 & \textrm{if }k\neq d+1, \\
 \ell  & \textrm{if }k= d+1,
\end{cases} 
\end{align}
for $i=1,\dots,d$.
Note that if $k=d+1$ and $\ell=m+1$, we obtain the sets $B\in \B$, so the previous display applies to the sets in $\B$ as well. In particular, the map $g_\zeta$ is independent of $\zeta$ and depends only on $F$ and $B$.

Let $L$ be the Laplacian of an $A$-compatible multidigraph $\mmG$ as in  Proposition \ref{prop:solCMT2}. Then $L$ agrees with $A$ on all rows but $j_1,\dots,j_d$, that is, on all rows but the ones with indices in $F\setminus\{m+1\}$ ($A$ is an $m\times m$ matrix). Therefore, for a set $B\subseteq \{1,\dots,m\}$ with $d$ elements, we have the following equality of minors
\begin{equation}\label{eq:AvsL}
A_{(F\setminus \{m+1\},B)}=L_{(F,B\cup \{m+1\})}. 
\end{equation}

\begin{lemma}\label{lemma:detA} With the notation introduced above, for $\ell=1,\ldots,m$,
\begin{align}
\det(A) & =(-1)^{m-d}\sum_{B\in \B}\left(\prod_{i=1}^d a_{j_i\beta_i}\right)\Upsilon_{\mmG}(F,B),\label{eq:detA}\\
\det(A_{\ell\rightarrow-b}) & = (-1)^{m-d}  \left(\sum\limits_{k=1}^d (-b_{j_k}) \sum\limits_{B\in \B^{k},\ell\notin B}\left(\prod\limits_{ i=1, i\neq k}^d a_{j_i\beta_i} \right) \Upsilon_{\mmG}(F,B\cup \{\ell\})\right.\label{eq:detAb}\\
&\quad \left.+\sum\limits_{B\in \B^{d+1},\ell\notin B}\left(\prod\limits_{i=1}^d  a_{j_i\beta_i} \right) \Upsilon_{\mmG}(F,B \cup \{\ell\}) \right). \nonumber
\end{align}
\end{lemma} 
\begin{proof} 
The nonzero entries of the $j_i$-th row of $A$ are in columns with index in $\mmN_i$. To prove \eqref{eq:detA} we consider the Laplace expansion of the determinant of $A$ along the rows $j_1, \dots,j_d$ \cite{FlexLaplace,Rose:LinAlg}. 
 Using \eqref{eq:AvsL} and Theorem \ref{thm:matrixtree} we have 
\begin{align*}
\det(A) & =\sum_{B\in \B^{d+1}}(-1)^{\sum_{i=1}^d (j_i+\beta_i)}\left(\prod_{i=1}^d \,  a_{j_i\beta_i} \right)A_{(F\setminus \{m+1\},B)} \\
&=(-1)^{m-d}\sum_{B\in \B}\left(\prod_{i=1}^d a_{j_i\beta_i}\right)\widetilde{\Upsilon}_{\mmG}(F,B).
\end{align*}
By \eqref{eq:rootsforests} with $k=d+1$ and $\ell=m+1$, if $\zeta\in\Theta_{\mmG}(F,B)$ with $B\in \B$, then $I(g_\zeta)=0$ and so $\widetilde{\Upsilon}_{\mmG}(F,B)=\Upsilon_{\mmG}(F,B)$, see \eqref{ypsilon}. This concludes the proof of \eqref{eq:detA}.

For $B\subseteq \{1,\dots,m+1\}$, 
we let  $$\varepsilon(\ell,B)=\big|\{i\in B\st \ell<i<m+1\}\big|,$$ 
and for $B\in\B^{k}$, $k\in\{1,\ldots,d+1\}$, we define 
$$\alpha_k(B)=\sum_{\begin{subarray}{c}i=1,\,  i\neq k\end{subarray}}^d(j_i+\beta_i)\quad\textrm{ and }\quad w_k(B)=\prod_{\begin{subarray}{c}i=1, \, i\neq k\end{subarray}}^d a_{j_i\beta_i}.$$

To prove \eqref{eq:detAb} we consider the Laplace expansion of the determinant of $A_{\ell\rightarrow-b}$ along column $\ell\in\{1,\ldots,m\}$. For $i\leq m-m_0$, we have $b_i=0$ if $i\neq j_k$ for all $k=1,\dots,d$. It gives
\begin{equation}\label{eq:detAb2}
\det(A_{\ell\rightarrow-b})=\sum_{k=1}^d (-b_{j_k})(-1)^{j_k+\ell} A_{(\{j_k\},\{\ell\})}+\!\!\!\sum_{k=m-m_0+1}^m(-b_k) (-1)^{k+\ell}A_{(\{k\},\{\ell\})}.
\end{equation}  
Fix  $k\in \{1,\dots,d\}$.  To compute $A_{(\{j_k\},\{\ell\})}$, we consider the Laplace expansion of the determinant of the submatrix $\widehat{A}$ of $A$, given by removing row $j_k$ and column $\ell$, along  the rows
$j_1,\dots$, $j_{k-1}$, $j_{k+1}-1,\dots,j_{d}-1$. These rows correspond to the rows $j_1,\dots$, $j_{k-1}$, $j_{k+1},\dots,j_{d}$ of $A$. 
Given $j\in  \{1,\dots,m\}$, the $j$-th column of $\widehat{A}$ is the $j$-th column of $A$ if $j<\ell$ and the $(j+1)$-th column if $j\geq \ell$.   
By \eqref{eq:AvsL} and Theorem \ref{thm:matrixtree} we have
\begin{align*}
A_{(\{j_k\},\{\ell\})} = &\sum_{B\in \B^{k}, \ell\notin B}(-1)^{\alpha_k(B)-(d-k)+\varepsilon(\ell,B)}  w_k(B)A_{(F\setminus \{m+1\},(B \setminus \{m+1\})\cup \{\ell\})}\\
= &\sum_{B\in \B^{k}, \ell\notin B} w_k(B)(-1)^{-d+k+\varepsilon(\ell,B)+m+1-d-1+\ell+j_k}\widetilde{\Upsilon}_{\mmG}(F,B\cup \{\ell\}).
\end{align*}

Let $B\in \B^{k}$, $k\in\{1,\ldots,d\}$ and $\ell \notin B$.  By Lemma \ref{lemma:rootsforests}(ii), $\Theta_{\mmG}(F,B\cup \{\ell\})=\emptyset$ 
if $\ell\in\mmN_i$, $i\not=k$ and $i>0$. By \eqref{eq:rootsforests},  we further have

\smallskip
\begin{itemize}
\item[-]  If $\ell\in\mmN_k$, then $I(g_\zeta)=0$ and $\varepsilon(\ell, B)=d-k$. 
\item[-]  If $\ell\in\mmN_0$, 
then $\varepsilon(\ell, B)=0$ and $I(g_\zeta)=d-k$ since there are $d-k$ inversions in $g_\zeta$: we have $j_i>j_k$ for $i>k$ and $g_\zeta(j_i)=\beta_i<\ell=g_\zeta(j_k)$.
\end{itemize}
Therefore,
\begin{equation}\label{eq:detAb3}
A_{(\{j_k\},\{\ell\})}=
(-1)^{\ell+j_k+m-d} \sum\limits_{B\in \B^k,\ell\notin B}\omega_k(B)\Upsilon_{\mmG}(F,B\cup \{\ell\}).
\end{equation}

Secondly, we find $A_{(\{k\},\{\ell\})}$ for $m \geq k>m-m_0$ similarly to above, by considering the Laplace expansion of the submatrix of $A$ obtained by removing row $k$ and column $\ell$, along the rows  $j_1, \dots,j_d$.  By \eqref{eq:AvsL} and Theorem \ref{thm:matrixtree} we obtain 
\begin{align*}
A_{(\{k\},\{\ell\})}&=\sum_{B\in \B^{d+1},\ell\notin B}(-1)^{\alpha_{d+1}(B)+\varepsilon(\ell,B)}w_{d+1}(B)A_{((F\setminus \{m+1\})\cup \{k\},B\cup \{\ell\})}\\
&=\sum_{B\in \B,\ell\notin B}w_{d+1}(B)(-1)^{ \varepsilon(\ell,B)+m+1-d-2+k+\ell}\widetilde{\Upsilon}_{\mmG}(F\cup \{k\},B\cup \{\ell\}).
\end{align*}
By Lemma~\ref{lemma:rootsforests}(ii),  for $\ell\leq m-m_0$, we have $\Theta_{\mmG}(F\cup \{k\},B\cup \{\ell\})=\emptyset$ if $B\in \B$, and also $\Theta_{\mmG}(F,B\cup \{\ell\})=\emptyset$ for $B\in \B^{d+1}$. Thus, from above
\begin{align}\label{eq:zeroB}
0 &= \sum_{k=m-m_0+1}^m(-b_k) (-1)^{k+\ell}A_{(\{k\},\{\ell\})}  \\
&= \sum\limits_{B\in \B^{d+1},\ell\notin B}\left(\prod\limits_{i=1}^d  a_{j_i\beta_i} \right) \Upsilon_{\mmG}(F,B \cup \{\ell\}).\nonumber
\end{align}

If $\ell>m-m_0$, then $\varepsilon(\ell, B)=0$. Using Lemma \ref{lemma:rootsforests}(i), we have for  $\zeta\in\Theta_{\mmG}(F\cup\{k\},B\cup \{\ell\})$ that $g_\zeta(j_i)=\beta_i$ for all $i\in \{1,\dots,d\}$, $g_\zeta(m+1)=m+1$ and $g_\zeta(k)=\ell$. Thus  
$I(g_\zeta)=0$ and
\begin{equation}\label{eq:detAb4}
A_{(\{k\},\{\ell\})}=(-1)^{\ell+k+m-d+1}\sum_{B\in \B,\ell\notin B}w_{d+1}(B)\, \Upsilon_{\mmG}(F\cup \{k\},B\cup \{\ell\}).
\end{equation}

It only remains to prove that for $\ell>m-m_0$, we have
\[
\sum\limits_{k=m-m_0+1}^m\!\!b_k\!\!\!\sum\limits_{B\in \B, \ell\notin B}\!\!\!w_{d+1}(B) \Upsilon_{\mmG}\big(F\cup \{k\},B \cup \{\ell\}\big) = \!\!\!\sum\limits_{B\in \B^{d+1},\ell\notin B}\!\!\!w_{d+1}(B) \Upsilon_{\mmG}\big(F,B \cup \{\ell\}\big),
\]
where the left side is \eqref{eq:detAb4} summed over $k$.
Note that for $B\in\B$, we have $m+1\in B$ and hence $B\setminus\{m+1\}\in\B^{d+1}$. Therefore it is sufficient to prove that for $B\in \B$ and $\ell\notin B$, we have
\begin{equation}\label{eq:detAb5}
\sum\limits_{k=m-m_0+1}^mb_k \Upsilon_{\mmG}\big(F\cup \{k\},B \cup \{\ell\}\big)= \Upsilon_{\mmG}\big(F,B \cup \{\ell\}\setminus \{m+1\}\big).
\end{equation}
Let $\mmE_{m+1,k}$ be the set of edges in $\mmG$ with source $m+1$ and target $k$. Note that $b_k=\sum_{e\in\mmE_{m+1, k}}\pi(e)$, so it is sufficient to show that 
\[
\bigcup_{k=m-m_0+1}^m \Big\{e\cup \zeta \st e\in\mmE_{m+1,k},\ \zeta\in \Theta_{\mmG}(F\cup \{k\},B \cup \{\ell\})\Big\}=\Theta_{\mmG}\big(F,B \cup \{\ell\}\setminus \{m+1\}\big),
\]
as each element $e\cup\zeta$ on the left side has label $\pi(e)$ times the label of the spanning tree $\zeta\in \Theta_{\mmG}\big(F,B \cup \{\ell\}\setminus \{m+1\}\big)$.

We will show the equality by proving that the left side is contained in the right side, and \emph{vice versa}.
Consider a spanning forest $\zeta\in \Theta_{\mmG}(F\cup \{k\},B \cup \{\ell\})$ and  $e\in\mmE_{m+1,k}$. Then one connected component of $\zeta$ is a tree rooted at $m+1$, and another a tree rooted at $\ell$ that contains $k$. 
In $\zeta\cup e$, these two connected components are merged into  a tree rooted at $\ell$ that contains $m+1$. Hence the inclusion $\subseteq$ holds.

 To prove the other inclusion, we note that  a spanning forest $\zeta\in\Theta_{\mmG}(F,B \cup \{\ell\}\setminus \{m+1\})$ contains exactly one edge $e$ with source $m+1$. 
Since $\mmG$ is $A$-compatible, the target of this edge belongs to $\mmN_0$ (see Figure~\ref{fig:struct}), that is, $e\in\mmE_{m+1,k}$ with  $m\geq k>m-m_0$. One connected component of the subgraph $\zeta\setminus \{e\}$ is a tree rooted at $m+1$ and another connected component  is a tree rooted at $\ell$ that contains $k$. So the desired inclusion  holds; hence the equality holds.

The proof of  equation \eqref{eq:detAb} now follows by combining the equations \eqref{eq:detAb2}--\eqref{eq:detAb5}. 
\end{proof}

By Cramer's rule,  the solution to the linear system is
$$x_\ell =\frac{\det(A_{\ell\rightarrow-b})}{\det(A)}.$$
Now, using Lemma~\ref{lemma:detA}, we obtain the expression in the statement of Proposition~\ref{prop:solCMT2}, 
after combining the two sums of \eqref{eq:detAb} into one using $b_{j_{d+1}}=-1$.

\subsection{Nonnegativity of the solution (Theorem \ref{thm:possol2})}\label{sec:positivity}
The aim of this section is to  prove  Theorem \ref{thm:possol2}, that is, to prove that the solution to the linear system \eqref{eq:secondsystem} is nonnegative under certain conditions. To do so, we prove that a decomposition, similar to the one  in \eqref{eq:decompTheta}, holds for $\Theta_{\mmG}(F,B)$ for certain subsets $B\subseteq \{1,\dots,m+1\}$.

Assume the multidigraph $\mmG$ is an $A$-compatible P-graph that satisfies condition (\hyperref[condpossol2]{*}) of Theorem \ref{thm:possol2}. In the lemmas below we consider 

\begin{center}
$B=\widetilde{B}\cup \{\ell\}$, where $\widetilde{B}\in \B^{k}$ for  $k\in \{1,\dots, d+1\}$ and $\ell\in \{1,\dots,m+1\}, \ell\notin \widetilde{B}$.
\end{center}

\begin{lemma}\label{lem:paths}
Let $\zeta\in \Theta_\mmG(F,B)$.
\begin{enumerate}[(i)]
\item Any path from $j_k\in F$, $k\in \{1,\dots,d\}$, to $i\in B$ in $\zeta$, does not contain an edge in $\mmE^-$.

\item If $\zeta'\in \Theta_\mmG(B)$ is such that the connected component of $\zeta'$ containing $j_k\in F$ has root $g_\zeta(j_k)$ for all $k=1,\dots,d$, then also the connected component containing $m+1$ has root $g_\zeta(m+1)$. In particular, $\zeta'\in \Theta_\mmG(F,B)$.
\end{enumerate}
 
\end{lemma} 
 \begin{proof}
 (i) By condition (\hyperref[condpossol2]{*}), any such path goes through $m+1$. But this implies that $j_k$ and $m+1$, which both are in $F$, also are in the same connected component of $\zeta$, contradicting the definition of $\Theta_{\mmG}(F,B)$.

(ii) Let $i\in B$ be the root of the connected component of $\zeta'$ containing $m+1$. 
By Lemma~\ref{lemma:rootsforests}(i), both $i$ and $g_\zeta(m+1)$ belong to $\mmN_0$.
If  $m+1\in B$, then  necessarily  $i=m+1=g_\zeta(m+1)$. Otherwise, by our choice of sets $B$,  $g_\zeta(m+1)$ is the only element both in $B$ and $\mmN_0$. Thus  it must hold that $i=g_\zeta(m+1)$.
 \end{proof}

\begin{lemma}\label{lemma:maxEF}
For $\zeta\in \Theta_{\mmG}(F,B)$, we define 
\[
\mmE^{F}_\zeta=\big\{E\subseteq \zeta\cap \im\mu \st (\zeta\setminus E)\cup \mu^{*}(E)\in \Theta_{\mmG}(F,B)\big\}\subseteq \mathcal{P}(\mmE^+).
\]
The set $\mmE_\zeta^F$ is closed under union, that is, if $E_1,E_2\in \mmE_\zeta^F$, then $E_1\cup E_2\in \mmE_\zeta^F$. 
\end{lemma}

\begin{proof}
Since $\Theta_{\mmG}(F,B)\subseteq \Theta_{\mmG}(B)$, we have  $\mmE^{F}_\zeta\subseteq \mmE_\zeta$,  and since $\mmE_\zeta$ is closed under union by Lemma \ref{lemma:maxEzeta}, we have $E_1\cup E_2\in \mmE_\zeta$ if  $E_1,E_2\in \mmE^{F}_\zeta$.  Hence $\zeta_3=(\zeta\setminus E_1\cup E_2)\cup \mu^{*}(E_1\cup E_2)$ is   a spanning forest with $d+1$ connected components, each with a root in $B$. We show that if $j\in F$ and $i\in B$ are in the same connected component of $\zeta$, then they are also in the same connected component of $\zeta_3$. By Lemma~\ref{lem:paths}(ii), it is enough to show this for $j\neq m+1$.

Consider the unique path from $j\neq m+1$ to $i$ (assuming $j\neq i$) in $\zeta$.  Assume this path contains an edge $e\in E_1\cup E_2$, say  $e\in E_1$.
Then $\zeta_1=(\zeta\setminus E_1)\cup \mu^{*}(E_1)$, which belongs to $ \Theta_{\mmG}(F,B)$ by hypothesis, has also a path from $j$ to $i$ by \eqref{eq:rootsforests}. Since every node different from the root of a rooted tree has exactly one outgoing edge, then $\mu^*(e)\in\mmE^-$ must belong to this path because it has source  $s(e)$.
But this contradicts Lemma~\ref{lem:paths}(i). Therefore  this path does not contain an edge in $E_1\cup E_2$ and hence it belongs  to $\zeta_3$. This implies that $j$ and  $i$ belong to the same connected component of $\zeta_3$.   This shows that $\zeta_3\in\Theta_{\mmG}(F,B)$.
\end{proof}

It is a consequence of the lemma that the set  $\mmE^{F}_\zeta$ has a unique maximum with respect to inclusion, which we denote by  $E^{F}_\zeta$.  Define
\begin{equation*}
\Lambda_\mmG(F,B)=\left\lbrace \zeta \in \Theta_\mmG(F,B)\st \mmE^{F}_\zeta=\emptyset\right\rbrace,
\end{equation*}
and a  surjective map  by
$$
\psi_F\colon \Theta_\mmG(F,B)\rightarrow\Lambda_\mmG(F,B),\quad  \psi_F(\zeta)=\Big(\zeta\setminus E^{F}_\zeta\Big)\cup \mu^{*}\Big(E^{F}_\zeta\Big).
$$
Then we have the following decomposition, analogous to the decomposition in \eqref{eq:decompTheta},
\begin{equation}\label{eq:decompThetaF}
\Theta_\mmG(F,B)=\bigsqcup\limits_{\zeta\in\Lambda_\mmG(F,B)}\psi^{-1}_F(\zeta).
\end{equation}

\begin{lemma}\label{lemma:PrepsiF}
For $\zeta\in \Lambda_{\mmG}(F,B)$, it holds
\[
\psi^{-1}_{F}(\zeta)=\left\{E\cup (\zeta\cap \mmE^+) \st E\in  \underset{e\in\zeta\cap \mmE^-}{\mathlarger{\mathlarger\odot}}(\{e\}\cup \mu(e))   \right\}.
\]
\end{lemma}

\begin{proof}
The inclusion $\subseteq$ is proven analogously to the proof of Lemma~\ref{lemma:decompPrePsi}. 
By Lemma~\ref{lemma:decompPrePsiEdges}, we know that the set on the right side consists of elements in $\Theta_{\mmG}(B)$. 
Further, if  $e\in\zeta\cap\mmE^-$ and $e'\in\mu(e)$, consider the forest $\zeta'=(\zeta\setminus \{e\})\cup \{e'\}.$ We will show that $\zeta'\in \Theta_{\mmG}(F,B)$.
Given $j\in F$, $j\neq m+1$, such that $g_\zeta(j)=i$, 
the path connecting $j$ and $i$ (if any) in $\zeta$ does not contain $e$ by Lemma~\ref{lem:paths}(i). Thus the path  is also in $\zeta'$ and $j,i$ are in the same connected component of $\zeta'$. 
Now by Lemma~\ref{lem:paths}(ii), $\zeta'\in \Theta_{\mmG}(F,B)$. 
Thus the set on the right is included in $\Theta_{\mmG}(F,B)$, and it is  straightforward to show that their image by $\psi_F$  is $\zeta$.
\end{proof}

We proceed analogously to the proof of Theorem~\ref{thm:posol}. 
By \eqref{eq:decompThetaF} and Lemma \ref{lemma:PrepsiF}, we obtain the following expression corresponding to \eqref{eq:positive}:
\begin{align*}
\Upsilon_\mmG(F,B)
&= \sum\limits_{\zeta\in\Lambda_\mmG(F,G)}\pi(\zeta\cap \mmE^+)\prod\limits_{e\in\zeta\cap \mmE^-}\left(\pi(e)+\sum\limits_{e'\in \mu(e)}\pi(e')\right). 
\end{align*}
By Definition \ref{def:Pgraph}(iv), 
we deduce that $\Upsilon_\mmG(F,B)\in R_{\geq 0}$. By Proposition~\ref{prop:solCMT2}, in particular \eqref{eq:solution2}, and the hypotheses of  Theorem \ref{thm:possol2}  on the signs of the entries of $A$ and $b$,  we conclude that Theorem \ref{thm:possol2} holds.
\qed

\begin{proof}[Proof of Corollary \ref{cor:extrazeros}]
By \eqref{eq:solution2} and the remark below Proposition~\ref{prop:solCMT2}, it is enough to show that  $\Theta_\mmG(F,B\cup \{\ell\})=\emptyset$  for  $B\in \mathcal{B}^i$, $i\in\{1,\ldots,d\}$. 
Assume thus that there exists $\zeta\in \Theta_\mmG(F,B\cup \{\ell\})$ with $B\in \mathcal{B}^i$. In particular $g_\zeta(j_i)=\ell$ by  \eqref{eq:rootsforests}  and hence there exists a path from $j_i$ to $\ell$ in $\zeta$. 
By the structure of an $A$-compatible multidigraph, all nodes in this path belong to $\mmN_i$. Thus the path can be extended to a path from $j_i$ to $j$ containing a node in $\mmE^-$, but not $m+1$, contradicting condition (\hyperref[condpossol2]{*}). 
\end{proof}

\setlength\bibsep{3pt}

\end{document}